\documentclass[a4paper, 11pt]{amsart}
\usepackage{amsmath, amssymb, amscd, amsthm, amsfonts, tikz, tikz-cd}
\usepackage{graphicx}
\usetikzlibrary{arrows.meta,positioning}
\tikzset{>=Stealth, dot/.style={circle,fill,inner sep=1.3pt}}
\usetikzlibrary{decorations.markings}
\tikzset{double line with arrow/.style args={#1,#2}{decorate,decoration={markings,%
mark=at position 0 with {\coordinate (ta-base-1) at (0,1pt);
\coordinate (ta-base-2) at (0,-1pt);},
mark=at position 1 with {\draw[#1] (ta-base-1) -- (0,1pt);
\draw[#2] (ta-base-2) -- (0,-1pt);
}}}}
\usepackage{stmaryrd}
\usepackage{comment}
\usepackage{mathrsfs}
\usepackage{hyperref}
\usepackage[T1]{fontenc}
\usepackage[french,british]{babel}
\usepackage[all]{xy}
\usepackage[margin=24.4mm]{geometry}

\title{A path model for MV polytopes in type $A_n$}

\author{Zijun Li}
\address{Institut Camille Jordan UMR 5208, Université Jean Monnet, CNRS, Centrale Lyon, INSA Lyon, Université Claude Bernard Lyon 1, 20, rue Annino, 42023, Saint-Étienne, France}
\email{zijun.li@univ-st-etienne.fr}

\date{\selectlanguage{british}\today}

\newtheorem{theorem}{Theorem}[section]
\newtheorem{lemma}[theorem]{Lemma}
\newtheorem{proposition}[theorem]{Proposition}
\newtheorem{corollary}[theorem]{Corollary}

\theoremstyle{definition}
\newtheorem{definition}[theorem]{Definition}

\newtheorem{example}[theorem]{Example}

\theoremstyle{remark}
\newtheorem{remark}[theorem]{Remark}
\newtheorem*{theorem*}{Theorem}

\newcommand{\Pol}{\operatorname{Pol}}
\newcommand{\V}{\operatorname{V}}
\newcommand{\F}{\operatorname{F}}
\newcommand{\rx}{\operatorname{r}}
\newcommand{\bx}{\operatorname{b}}

\newcommand{\W}{\mathcal{W}}
\newcommand{\SC}{\mathcal{SC}}
\newcommand{\Ir}{\operatorname{Ir}}
\newcommand{\Irr}{\operatorname{Irr}}

\newcommand{\B}{\mathcal{B}}
\newcommand{\CC}{\mathcal{C}_A}
\newcommand{\md}{\operatorname{mod}} 
 
\newcommand{\id}{\operatorname{id}} 
\newcommand{\Z}{\mathbb{Z}}
\newcommand{\wt}{\operatorname{wt}}
\newcommand{\MV}{\mathcal{MV}}
\newcommand{\In}{\mathbf{i}}
\newcommand{\Conv}{\operatorname{Conv}}
\newcommand{\ASC}{\mathcal{A}_\mathcal{B}}
\newcommand{\Fp}{\mathcal{F}_p}
\newcommand{\CN}{\mathbb{C}[N]}

\newcommand{\rn}{\frac{(n+1)n}{2}}
\newcommand{\C}{\mathbb{C}}
\addto\captionsfrench{}
\addto\captionsfrench{}

\addto\captionsfrench{}
\begin{document}

\begin{abstract}
We introduce a one-skeleton path model for Mirković-Vilonen polytopes in type $A_n$. We prove that the Minkowski sum of polytopes corresponds to the concatenation of one-skeleton paths of this model. We show that MV polytopes induced by fundamental one-skeleton paths are Harder–Narasimhan polytopes. The paths determined by an orientation of the edges of the fundamental alcove precisely parameterize the cluster variables in the initial seed of $\CN$.\par
We also establish a correspondence between fundamental one-skeleton paths and folded galleries representing maximal faces of subword complexes. Under this correspondence, the comultiplication structure of $\CN$ matches the intrinsic comultiplication structure of folded galleries given by projections to sub-Coxeter complexes.

\end{abstract}
\maketitle
\renewcommand{\contentsname}{Contents}

\section{Introduction}

\subsection{Crystals and perfect bases}
Given a simply-laced simple algebraic group $G$ over $\mathbb{C}$, the set of irreducible representations of $G$ is parametrized by the dominant weights of the root system associated with $G$. Any such irreducible module $L(\lambda)$ can be embedded into $\CN$, the coordinate ring of the unipotent radical $N$ by an $N$-equivariant map $\psi _{\lambda}$ (see \cite{MVHD}). \par

A $G$-crystal $B$ is a combinatorial model encoding a representation of $G$  through crystal operators and the weight function. A ($G$-)bicrystal is a set with two ($G$-)crystal structures sharing the same weight function.  
\par

A perfect basis of $\mathbb{C}[N]$ is a linear basis $\{ v_{\alpha}\}_{\alpha \in B(\infty)}$ with a $G$-crystal parameterization $B(\infty)$ such that the action of the Chevalley generators is controlled by the $G$-crystal structure. Similarly, a biperfect basis of $\mathbb{C}[N]$ is a linear basis $\{ v_{\alpha}\}_{\alpha \in B(\infty)}$ with a $G$-bicrystal structure. For the explicit definition, see \cite{MVHD}. Any biperfect basis is naturally a perfect basis. \par

Perfect bases, and even biperfect bases, of $\mathbb{C}[N]$ exist for any $G$ of finite type. Moreover, the biperfect basis of $\mathbb{C}[N]$ is unique when $G$ is $SL_2$, $SL_3$ or $SL_4$ \cite{MVHD} \cite{PERF}. The first example of a biperfect basis is Lusztig's dual canonical basis \cite{lusz} \cite{Ting1}. The Mirković–Vilonen basis arising from the affine Grassmannian and the dual semicanonical basis arising from preprojective algebras are two other constructions that yield biperfect bases \cite{proj1} \cite{affin}.

\subsection{MV polytopes and path models}

The set of (stable) Mirković-Vilonen (MV) polytopes is a combinatorial realization of the $G$-crystal $B(\infty)$ \cite{JOEL}. An MV polytope is a polytope lying in the weight space $X=P\otimes \mathbb{R}$ with integral vertices satisfying tropical Plücker relations. The action of crystal operators can also be described using moves of integral vertices in the directions of simple roots. \par
MV polytopes arise from different (bi)perfect bases mentioned above. MV polytopes are precisely the images of the MV cycles under the moment map \cite{JOEL} \cite{GL}. The dual semicanonical basis is parametrized by the irreducible components of the variety of representations of the associated preprojective algebra \cite{BKT} \cite{PERF} \cite{proj1}. MV polytopes can also be obtained from dimension vectors of simple modules of the Khovanov-Lauda-Rouquier algebras \cite{KLR}. The operations on these bases correspond to operations on MV polytopes. For example, Baumann, Kamnitzer, and Knutson prove that the product of two functions $b_P$ and $b_{P'}$ in the MV basis corresponds to the convolution of Duistermaat–Heckman measures $\mu _p$ and $\mu _{p'}$ in \cite{MVHD}. The corresponding operation on MV polytopes is the Minkowski sum of polytopes. \par


Moreover, MV polytopes can also be retrieved from the comultiplication structure of $\CN$. Given a function $f\in \CN$ with a minimal writing of the comultiplication $\Delta(f)=\sum_{i=1}^n b_i\otimes c_i$, Baumann constructs a polytope associated to $f$ using the convex hull of all weights of $b_i$, and shows that we can recover the MV polytope using this construction in \cite{onmv}. \par

MV polytopes are associated with galleries in buildings by the root operators defined by Littelmann in \cite{LP}. Baumann, Gaussent and Littelmann use the Lakshmibai–Seshadri galleries (LS galleries) to study finite-dimensional representations of $G$ in \cite{BG}. Their construction can be generalized to the affine case using masures \cite{GL}. Ehrig establishes the connection between LS galleries and MV polytopes \cite{ehrig}. The action of crystal operators on the family of LS galleries provides a combinatorial method to construct the MV polytope associated with a given LS gallery. The work \cite{OSGL} by Gaussent and Littelmann gives a one-skeleton model of LS galleries, which can also be used to compute the associated MV polytopes. \par

However, there is no natural operation on LS galleries corresponding to the Minkowski sum of polytopes. For any path in the weight space $X$, the associated polytope using Littelmann's crystal operators is always an MV polytope \cite{BG}. Since the Minkowski sum of two MV polytopes is not, in general, an MV polytope \cite{JOEL}, we cannot expect to find a corresponding operation using LS galleries. \par
To solve this problem, we construct another path model with associated crystal operators in type $A_n$ in this paper. 
We prove that the Minkowski sum of polytopes corresponds to the concatenation of one-skeleton paths in this model. The first main theorem in this paper is the following property:
\begin{theorem*}
    For two one-skeleton paths $p$ and $q$, we have\begin{align}
        \Pol(p*q)=\Pol(p)\boxplus \Pol(q), \nonumber
    \end{align} where $\Pol$ is a map from one-skeleton paths to polytopes, $*$ is the concatenation of paths and $\boxplus$ is the Minkowski sum of polytopes.
\end{theorem*}

We prove that the MV polytopes arising from fundamental one-skeleton paths coincide with those associated with 1-filtered indecomposable modules of the preprojective algebra. 
We also show that paths given by an orientation of edges of the fundamental alcove provide cluster variables in the initial seed of $\CN$.\par

The path model and root operators introduced by Littelmann in Section 1 of \cite{LP} yield $G$-crystal structures. In contrast, our one-skeleton path model induces crystal structures only for length-one paths. The advantage of our framework is that concatenation of paths is compatible with the Minkowski sum of polytopes. Moreover, our model also allows us to compute the comultiplication coefficients. \par
As a side note, we remark that Kamnitzer gives a geometric criterion for a polytope being an MV polytope using 2-faces \cite{JOEL}. This result is generalized to the affine case in \cite{affin} and \cite{BKT}. 

\subsection{Subword complexes, proto-exact categories and Hall algebras}

We further relate this path model to subword complexes through folded galleries. 
More precisely, certain fundamental one-skeleton paths can be represented by folded galleries associated with maximal faces of suitable subword complexes. 
This realization allows us to compare the comultiplication of \(\mathbb C[N]\) with the intrinsic decomposition of folded galleries given by projections to sub-Coxeter complexes.\par
The subword complex $\SC_W(Q,\pi)$, introduced by Knutson and Miller \cite{subword}, is a simplicial complex associated with a word $Q=s_{i_1}\cdots s_{i_m}$ consisting of simple reflections of a Coxeter group $W$ and an element $\pi\in W$. A face of $\SC_W(Q,\pi)$ is given by a set $I$ of $[m]$ such that the subword of $Q$ obtained by removing positions in $I$ contains a reduced expression of $\pi$. Each maximal face of a subword complex can be represented by a folded gallery in the corresponding Coxeter complex. We refer the reader to \cite{cls, gorsky_edge, G_subword_3} for the study of the effect of changing reduced expressions of $\pi$ on $\SC_W(Q,\pi)$. We refer the reader to \cite{escobar, CGGLSS,CGGS2,GLSB, CGGSSBS,trinh2021hecke,casals_gao} for the relations between subword complexes and Bott-Samelson varieties and cluster algebras. \par

Exact categories, which appear naturally in representation theory and algebraic geometry, are additive categories equipped with a distinguished class of admissible short exact sequences; every abelian category is exact in this sense. 
Proto-exact categories, introduced by Dyckerhoff and Kapranov in \cite{DK}, are non-additive analogues of exact categories and arise in examples such as matroids, representations over $\mathbb{F}_1$, and Banach algebras.\par
For a finitary abelian category $\mathcal{L}$, Ringel defined the Hall algebra $\mathcal{H}(\mathcal{L})$ in \cite{R1}: it is spanned by isomorphism classes of objects, with multiplication given by extensions. 
A classical example is the Hall algebra of nilpotent representations of a finite quiver over $\mathbb{F}_q$, which realizes the nilpotent part of the quantum group with parameter $\sqrt{q}$ \cite{R1}; see \cite{Hu, schiffmann1, schiffmann2, schiffmann3} for further developments. 
The 2-Segal formalism of \cite{DK} extends Hall (co)algebras to proto-exact categories. 
In \cite{GorskyLi}, we show that certain proto-exact categories arising from subword complexes have Hall algebras that are quotient algebras of universal enveloping algebras.\par
For a fixed standard object $X$ represented by the gallery $g_X$ of type $A_n$, we define a root configuration quiver $\Gamma_X$ associated with $X$ in the joint work with M. Gorsky \cite{GorskyLi}. We construct a subquiver category $\mathcal{S}_X$, whose objects are finite disjoint unions of quivers associated to projections of $g_X$. We prove that $\mathcal{S}_X$ is a proto-exact category. Thus $\mathcal{S}_X$ has a Hall algebra $\mathcal{H}(\mathcal{S}_X)$, in the sense of \cite{DK}. This Hall algebra is a symmetry-breaking version of the Hopf algebra defined in \cite{BC}. \par
We prove that the Hall coalgebra $\mathcal{H}^*(\mathcal{S}_X)$ is isomorphic to $\CN$ in type $A_n$. We prove that the comultiplication structure of $\CN$ coincides with that of $\mathcal{H}^*(\mathcal{S}_X)$ given by projections of galleries intrinsically. This correspondence provides a geometric interpretation of the comultiplication structure of $\CN$. \par

In the joint work with M. Gorsky \cite{GorskyLi}, we study properties of the subquiver category $\mathcal{S}_X$ for any $X$ with a root configuration quiver $\Gamma_X$ of tree type. Any fundamental one-skeleton path associated with an indecomposable module of an arbitrary orientation of the Dynkin diagram can be represented by a folded gallery. These galleries can be obtained from the standard folded gallery via a sequence of flips, corresponding to reflections of the root configuration quiver \cite{GorskyLi}.

\subsection{Outline}

Section~\ref{sec: pre} summarizes the preliminaries on crystals, MV polytopes and subword complexes. We construct the one-skeleton model and define crystal operators and the polytopal map in Section~\ref{sec:osp}. In Section~\ref{sec:mvpolytopes}, we construct the correspondence between fundamental one-skeleton paths and 1-filtered indecomposable modules of the preprojective algebra. We prove that cluster variables in the initial seed are given by edges of the fundamental alcove in Section~\ref{sec: cluster variables}. In Section~\ref{sec: comulti}, we construct the subquiver category and prove the isomorphism between the Hall coalgebra $\mathcal{H}^*(\mathcal{S}_X)$ and $\CN$. In Section~\ref{sec: examples}, we give examples in the low-rank cases.

\section{Preliminaries}\label{sec: pre}

\subsection{Perfect basis and crystals}\label{subsection 2.1}
Throughout this paper, we denote a simply-laced simple algebraic group of type $A_n$ by $G$. Let \(B\), \(N\), and \(T\) denote a Borel subgroup, its unipotent radical, and a maximal torus of \(G\), respectively. The weight lattice of $G$ is denoted by $\Lambda$ and $X=\Lambda \otimes \mathbb{R}$ denotes the weight space.  For any positive integer $m<n$, the set $\{1,\cdots, m\}$ is denoted by $[m]$ and the set $\{m,\cdots, n\}$ is denoted by $[m,n]$. For any $\mathbb{C}$-algebra $A$, we denote the category of finite dimensional left modules of $A$ by $mod(A)$.\par
The Weyl group $N_G(T)/T$ is denoted by $W$, which is isomorphic to the symmetric group $\mathfrak{S}_{n+1}$. We denote the longest element of $W$ by $w_0$. The root system is denoted by $\Phi$. The set of simple roots of $\Phi$ is denoted by $\{\alpha_i\}_{i\in [n]}$. The simple reflection with respect to $\alpha_i$ is denoted by $s_i$. The Weyl group $W$ is generated by $\{s_i\}_{i\in [n]}$. The dual of a simple root $\alpha_i$ is denoted by $\alpha_i^\vee$. The set of dominant weights is denoted by $\Lambda_+ := \{\lambda\in \Lambda \mid 
\langle \lambda, \alpha_i^\vee\rangle\ge 0 \text{ for all } i\}$. The fundamental weight $\omega_i$ is determined by the canonical pairing $\langle\omega_i,\alpha_j^\vee\rangle=\delta_{ij}$. We denote the Cartan matrix of $W$ by $(a_{i,j})_{1\leq i,j\leq n}$. In type $A_n$, we have \begin{align}
\langle \alpha_i,\alpha_j^{\vee}\rangle=a_{i,j}=\begin{cases}
    2 & i=j, \\
    -1 & |i-j|=1, \\
    0 & \text{else}. 
\end{cases} \nonumber\end{align}
  \par

The set of irreducible representations of $G$ is parametrized by the dominant weights in $X$. Any such irreducible module $L(\lambda)$ can be embedded into $\mathbb{C}[N]$, the algebra of regular functions on $N$ by an $N$-equivariant map $\psi _{\lambda}$ (see \cite{MVHD}). \par

The concept of the $G$-crystal is central in the study of irreducible representations of $G$. 
\begin{definition}\label{def:crystal}
A set $B$ is called a \textbf{G-crystal} if 0 is not in $B$ and it is equipped with maps:
       \begin{align}
           &\operatorname{wt}: B \to \Lambda\ , \ \varepsilon _i : B \to \mathbb{Z}\cup \{-\infty\}\ ,\ \varphi _i : B \to \mathbb{Z}\cup \{-\infty\}\ ,  \nonumber \\ & \ e_i : B\to B\cup \{0\} \,\ \text{ and } f_i : B\to B\cup \{0\}, \nonumber
       \end{align}
       satisfying the following conditions:
       \begin{enumerate}

\item For each $b \in B$ and $i \in [n]$, $\varphi_i(b) = \langle \alpha_i^\vee, \operatorname{wt}(b) \rangle + \varepsilon_i(b)$.

\item For each  $b, b' \in B$ and $i \in [n]$, we have $b = e_i b' \iff f_i b = b'$.

\item For each $b \in B$ and $i \in [n]$ such that $e_i b \neq 0$, we have 
\[
\operatorname{wt}(\tilde{e}_i b) = \operatorname{wt}(b) + \alpha_i, 
\]
\[
\varepsilon_i(e_i b) = \varepsilon_i(b) - 1, \quad \text{and} \quad \varphi_i(e_i b) = \varphi_i(b) + 1.
\]
\item For each $ b \in B$ and $i \in [n]$, if $\varphi_i(b) = \varepsilon_i(b) = -\infty $, then $e_i b = f_i b = 0$.
       \end{enumerate}
\end{definition}       
A ($G$-)bicrystal is a set with two ($G$-)crystal structures sharing the same weight function. A $G$-crystal isomorphism between two $G$-crystals is a bijection which preserves the $G$-crystal structure. In this paper, if a set $\{x_t\}_{t\in B}$ is parametrized by a $G$-crystal $B$, we will, without ambiguity, let the crystal structure maps of $B$ act directly on the elements $x_t$. For example, we use $e_i(x_t)$ to represent the element $x_{e_i(t)}$. 
\par

Let $\{\bar{e}_i\}_{i\in [n]}$ and $\{\bar{f}_i\}_{i\in [n]}$ denote the Chevalley generators of $G$. We have natural left and right actions of $\bar{e}_i$ and $\bar{f}_i$ on $\mathbb{C}[N]$. A biperfect basis of $\mathbb{C}[N]$ is a linear basis $\{ v_{\alpha}\}_{\alpha \in B(\infty)}$ parametrized by a $G$-bicrystal $B(\infty)$, which is compatible with the family of embeddings $\psi _{\lambda}$ such that the leading term of $\bar{e}_i \cdot v_{\alpha}$ is determined by $\varepsilon _i(v_{\alpha})$ and $e_i(v_{\alpha})$,  along with similar conditions for $f_i (v_{\alpha})$, $\varphi _i(v_{\lambda})$ and the right action of Chevalley generators. For the explicit definition, see \cite{MVHD}. When we only consider the crystal structure of $B(\infty)$ associated with the left action of Chevalley generators, we have a $G$-crystal and a perfect basis of $\CN$. By Theorem 7.2.2 in \cite{kashiwara}, as a $G$-crystal, $B(\infty)$ has a unique automorphism given by the identity map $\operatorname{id}$. \par

\subsection{Mirković-Vilonen polytopes}

Given any two sets $B$ and $B'$ parametrized by the same $G$-crystal $B(\infty)$, there is a unique bijection between $B$ and $B'$ which preserves the $G$-crystal structure. The set of (stable) MV polytopes is a set of polytopes which is parametrized by the $G$-crystal $B(\infty)$. An integral convex polytope in $X$ is the convex hull of finitely many points in $\Lambda$. The partial order $\preceq$ for points in $X$ is defined by $x\preceq y$ if and only if $y-x=\sum_{i\in I}a_i\alpha_i$, where $a_i\geq 0$ for all $i\in [n]$. \par

For any reduced expression $s_{i_1}\cdots s_{i_N}$ of $w_0$, we set $\mathbf{i}=(i_1,\cdots , i_N)$ to be the type of this reduced expression. The $\mathbf{i}$-Lusztig datum of a polytope $P$ lying in $X$ is a sequence of integers $\mathbf{a}_{\mathbf{i}}=(a_1,\cdots ,a_N)$ such that $\{v_i\}_{0\leq j\leq N}\subseteq P$, where $v_0=\mathbf{0}$ and $v_j=v_{j-1}+a_j\alpha_{i_j}$ for any $1\leq j\leq N$. 
\begin{theorem}\label{mvpolytope}[Theorem 7.1 in \cite{JOEL}]
    For a fixed $G$, there exists a family of integral convex polytopes in $X$ $\MV=\{P_i\}_{i\in B(\infty)}$ parametrized by the $G$-crystal $B(\infty)$ satisfying the following conditions:
    \begin{enumerate}
        \item Any $P\in \MV$ has a maximal vertex $\mu$ and a minimal vertex $\textbf{0}$ such that $\textbf{0} \preceq x \preceq \mu$ for any $x\in P$. The weight $\mu$ is called the weight of $P$. 
        \item Any sequence $(a_1,\cdots a_N)\in \mathbf{N}^N$ is the $\mathbf{i}$-Lusztig datum of a unique polytope $P\in \MV$ with weight $\sum_{1\leq j\leq N}a_j\alpha_{i_j}$.
        \item For any $j\in I$, $P\in\MV$ and a type $\mathbf{i}'=(i'_1,\cdots , i'_N)$ of a reduced expression of $w_0$ such that $i'_N=j$, $f_{j}(P)=0$ if the $\mathbf{i}'$-Lusztig datum $(a_1,\cdots ,a_N)$ of $P$ satisfies $a_N=0$. Otherwise, $f_{j}(P)$ is the unique MV polytope with $\mathbf{i}'$-Lusztig datum $(a_1,\cdots ,a_N-1)$.
    \end{enumerate}
\end{theorem}
We can also calculate all MV polytopes using Theorem 3.1 in \cite{JOEL}. In this paper, we consider polytopes up to translation: two polytopes are regarded as equivalent if one can be obtained from the other by a translation. \par

For two polytopes $P_1$ and $P_2$ in $X$, the \textbf{Minkowski sum} of $P_1$ and $P_2$ is defined as $P_1\boxplus P_2:=\{x+y\mid x\in P_1, y\in P_2\}$. We say that an MV polytope is prime if it is not the Minkowski sum of two smaller MV polytopes. The following theorem given by Kamnitzer provides a finite generating set of MV polytopes under Minkowski sum.

\begin{theorem}\label{summv}[Theorem 6.2 in \cite{JOEL}]
    For a fixed $G$, there exist finitely many prime MV polytopes in $\MV$. Any MV polytope is the Minkowski sum of finitely many prime MV polytopes.
\end{theorem}

\subsection{Subword complexes}
Given an arbitrary Coxeter group $W$ of rank $n$ with generators $\{s_i\}_{i\in [n]}$, a word $Q=s_{i_1}\cdots s_{i_m}$ and an element $\pi\in W$, the subword complex $\SC_W(Q,\pi)$ is the simplicial complex whose faces are subsets $I$ of $[m]$ such that the subword of $Q$ with positions at $[m]\backslash I$ contains a reduced expression of $\pi$.  

We consider the set of all quadruples $\tilde{\mathcal{C}}=\{(W,Q,\pi ,I)\}$, where $W$ is an arbitrary Coxeter group and $I$ is a maximal face of $\SC_W (Q,\pi)$. We denote the length of the word $Q$ by $l(Q)$. The maximal face $I$ can be viewed as a subset of $[l(Q)]=\{1,\cdots,l(Q)\}$. For any quadruple $X$ in $\tilde{\mathcal{C}}$, we usually use the subscript $X$ to denote the component corresponding to $X$ in the quadruple. In other words, $X=(W_x,Q_X,\pi _X,I_X)$. The vector space spanned by the roots of $W_X$ is denoted by $V_X$. The length of the word $Q_X$ is denoted by $n_X$. The root system associated with $X$ is denoted by $\Phi_X$. \par

\begin{definition}\label{coxetercomplex}
Let $(W,S)$ be a Coxeter system with $S=\{s_1,\dots,s_n\}$.  
The \textbf{Coxeter complex} $\Sigma(W)$ is the simplicial complex whose simplices are the right cosets $wW_J$ for subsets $J\subseteq S$,
ordered by reverse inclusion.  
A maximal simplex (a \textbf{chamber}) has the form $wW_{\emptyset}=\{w\}$, so chambers are naturally in bijection with elements in $W$.
\end{definition}

\begin{definition}\label{subcomplex}
For a subset $J\subseteq S$, the parabolic subgroup $W_J$ is the subgroup of $W$ generated by $\{s_i\}_{i\in J}$. The \textbf{sub-Coxeter complex associated to $J$} is the subcomplex
\[
  \Sigma_J(W)
  := \{\, w W_I \;\mid\; I\subseteq J \,\}
  \subseteq \Sigma(W).
\]
It is canonically isomorphic to the Coxeter complex $\Sigma(W_J)$.
\end{definition}

Let $V$ be the real vector space spanned by the simple roots 
$\{\alpha_s \mid s\in S\}$, equipped with the pairing 
$\langle\cdot,\cdot\rangle$ such that $\langle\alpha_i,\alpha_j^{\vee}\rangle=-cos(\frac{\pi}{\operatorname{ord}(s_is_j)})$.  In the crystallographic case, we use the standard Cartan pairing.\par

A chamber of the Coxeter complex $\Sigma(W)$ may be identified with a Weyl chamber
\[
  C(w) := \{\, x\in V \mid \langle x, w(\alpha_s^{\vee})\rangle > 0 
      \text{ for all } s\in S \,\}.
\]

If $C(w)$ and $C(ws)$ are adjacent chambers, they are separated by the wall
\[
  H_{w\alpha_s} := 
  \{\, x\in V \mid \langle x, w(\alpha_s^{\vee})\rangle = 0 \,\}.
\]

The \textbf{type} associated to this adjacency from $C(w)$ to $C(ws)$ is
\[
  \alpha(C(w),C(ws)) := w(\alpha_s),
\]
the normal vector orthogonal to the separating wall $H_{w\alpha_s}$. 

\begin{definition}\label{def:gallery}
A \textbf{gallery} (starting at the fundamental chamber $id$) in $\Sigma(W)$ is a finite sequence of chambers
\[
g=(C_0,C_1,\dots,C_k)
\]
such that $C_{i-1}$ and $C_i$ are adjacent for all $i$ and $C_0=id$. If we identify a chamber with the corresponding Weyl chamber, we also denote the gallery $g$ by the sequence of types
\[[\alpha(C_0,C_1),\cdots \alpha(C_{r-1},C_r)].\]
\end{definition}

\begin{definition}\label{def:foldedgallery}
A \textbf{folded gallery} (starting at the fundamental chamber $id$) is a sequence of chambers
\[
g=(C_0,C_1,\dots,C_k)
\]
such that for each $j$ one has either:
\begin{enumerate}
    \item $C_j \neq C_{j+1}$ and $C_{j+1},C_j$ are adjacent of type $\alpha(C_{j},C_{j+1})$ (\textbf{traversing position}); or
    \item $C_j = C_{j+1}$ with a choice of type of adjacency $\alpha(C_{j},C'_j)$, where $C'_j$ is a certain chamber adjacent to $C_j$ (\textbf{reflection position}). By a slight abuse of notation, we also denote this type by $\alpha(C_{j},C_{j+1})$.
\end{enumerate}
We set $I:=\{ i\mid 1\leq i \leq k, C_i=C_{i-1}\}$. If we identify a chamber with the corresponding Weyl chamber, we also denote the folded gallery $g$ by the sequence of types \[
[(-1)^{\delta_I(1)}\alpha(C_0,C_1),\cdots ,(-1)^{\delta_I(r)}\alpha(C_{r-1},C_r),I].\]
Here $\delta_I$ is the indicator function of $I$. We use $\delta_I$ to distinguish the traversing positions ($\delta_I(i)=0$) and the reflection positions ($\delta_I(i)=1$).
\end{definition}

This definition is well suited for our purpose and proof in Section~\ref{sec: comulti}. In this paper, we always use the sequence of types to represent a (folded) gallery. 

The geometric projection of a folded gallery $g$ in $\Sigma(W)$ to a sub-Coxeter complex $\Sigma(W_J)$ is obtained by extracting all roots in the expression of $g$ contained in the subroot system $\Phi_J$.

\section{One-skeleton path model}\label{sec:osp}
In this section, we first define one-skeleton paths and crystal operators. Then we introduce the polytopal map, which sends an arbitrary one-skeleton path to an integral polytope. We show that the polytopal map sends the concatenation of one-skeleton paths to the Minkowski sum of polytopes.

\subsection{One-skeleton paths}\label{subsction: one-skeleton paths}
The weight space $X$ has a partition into alcoves under the action of the affine Weyl group. The fundamental alcove is the set $\Delta=\operatorname{Conv}(0,\omega_1,\ldots,\omega_n)
=\left\{\sum_{i=1}^n a_i\omega_i \ \middle|\ a_i\ge 0,\ \sum_{i=1}^n a_i\le 1\right\}$. The Weyl group $W$ is isomorphic to the symmetric group $S_{n+1}$. The affine Weyl group $W^a=W \ltimes \Lambda$ acts transitively on the set of alcoves in $X$. An alcove is the set $g\cdot \Delta$ for some $g\in W^a$. An edge of an alcove is just one of its one-dimensional faces. We use the edges of alcoves to define one-skeleton paths.
\begin{definition}\label{def: one-skeleton path}
    A \textbf{one-skeleton path} (of type $A_n$) in $X$ is a piecewise linear function $p$ from $[0,1]$ to $X$ such that the image of $p$ is entirely contained in the union of edges of alcoves in $X$, $p(0)=0$ and $p(1)\in \Lambda$. A \textbf{positive fundamental one-skeleton path} in $X$ is a linear function from $[0,1]$ to $X$ associated with a fundamental weight $\omega_i$: \begin{align}
        p_i:t\to t \cdot \omega_i. \nonumber
    \end{align}
\end{definition}
We also use $p$ or $p_i$ to represent its image in $X$. 
\begin{definition}\label{def: concatenation of paths}   
For two one-skeleton paths $p_1$ and $p_2$ from $[0,1]$ to $X$, the concatenation $p_1 *p_2$ is the one-skeleton path
\[
p_1*p_2(t) =
\begin{cases}
p_1(2t) & t\in [0,\frac{1}{2}] \\
p_1(1)+p_2(2t-1)  & t\in [\frac{1}{2},1].
\end{cases}
\]
\end{definition}

\begin{remark}\label{rem:equivalent paths}
    We always view two one-skeleton paths as the same if they just differ by an increasing piecewise linear bijection from $[0,1]$ to $[0,1]$. Precisely,  we also use $p$ to denote the equivalence class of one-skeleton paths
    \[\{f \mid f=p\circ \tau\text{, where } \tau \text{ is an increasing piecewise linear bijection from } [0,1] \text{ to }  [0,1]\}.\]
    Up to this equivalence, the concatenation operation is associative.
\end{remark}

We denote by $OP$ the set of all one-skeleton paths of type $A_n$. \par

Recall that any simple root $\alpha_i$ gives a reflection $s_i$ on $X$ defined by \[
s_i(v)=v-\langle v,\alpha _i ^{\vee}\rangle \cdot \alpha _i,
\]
 where $\langle \ , \ \rangle$ is the canonical pairing defined in~\ref{subsection 2.1}. We define fundamental one-skeleton paths and the crystal operators on fundamental one-skeleton paths inductively.
\begin{definition}\label{def: crystal operator on fundamental}
   We let 0 denote a null path without an endpoint. We let $P_{f,0}$ denote the set of positive fundamental one-skeleton paths. For any $j\geq 0$, we define a set $P_{f,j}$ and crystal operators $f_i$  from $P_{f,j}$ to $P_{f,j}\cup \{0\}$ for any $i\in [n]$ inductively:
    \begin{align}
        f_i(p) =
\begin{cases}
s_i(p) & \textit{if } \langle p(t) ,\alpha_i ^{\vee} \rangle >0, \forall t\in (0,1]. \\
0  & else,
\end{cases} \nonumber
    \end{align}
    and $P_{f,j+1}:=\{f_i(p)\mid p\in P_{f,j}, i\in [n], f_i(p)\neq 0\}$. Let $P_f=\bigcup _{j\geq 0}P_{f,j}$. Any one-skeleton path in $P_f$ is called a \textbf{fundamental one-skeleton path} of type $A_n$.
\end{definition}
Given $p_1$ and $p_2$ in $P_f$ such that $f_i(p_1)=f_i(p_2)\neq 0$ for some $i\in [n]$, we have $p_1=s_i^2(p_1)=s_i^2(p_2)=p_2$. Thus the operator $f_i$ is injective on set $f^{-1}_i(P_f)$. We define the crystal operator $e_i$ as the inverse of $f_i$:
\begin{definition}\label{def: inverse fundamental operator}
    For any $i\in [n]$, the crystal operator $e_i$ on $P_f$ is defined by the following rule:
    \begin{align}
        e_i(p) =
\begin{cases}
s_i(p) & \textit{if } p \in P_f  \ \textit{and}  \ p=f_i(q) \  \textit{for some} \ q\in P_f ,\\
0  & else.
\end{cases} \nonumber
    \end{align}
\end{definition}
The operator $e_i$ is well defined since the operator $f_i$ is injective on set $f^{-1}_i(P_f)$.\par
We let $\W$ denote the monoid of words consisting of $f_i$. For any $w=f_{i_1}\cdots f_{i_m}\in \W$ and $p\in P_f$, the element $f_{i_1}\circ \cdots \circ f_{i_m}(p)$ is denoted by $w(p)=w\cdot p$ in $P_f$.

\begin{definition}\label{def: j-chain}
The \textbf{j-chain} of $G$ (of type $A_n$) is an $n$-colored quiver $Q_j$ with vertex set $V_j=\{ w\cdot \omega_j\}_{w\in \mathcal{W},w\cdot \omega_j\neq 0}$, for any $j\in [n]$. For two vertices $v_1$ and $v_2$ in $V_j$, there is an $i$-arrow from $v_1$ to $v_2$ if and only if $v_2=f_i(v_1)$. Notice that $P_f=(\bigsqcup _{j\in [n]}V_j)\sqcup \{0\}$. 
\end{definition}
For any one-skeleton path $p$, the weight of $p$ is defined by the endpoint of $p$.\[
\wt(p):=p(1)\in P.
\] The weight of the null path $0$ is defined by $\wt(0)=-\infty$. The following proposition indicates that $V_j$ has a $G$-crystal structure for any $j\in [n]$. 
\begin{proposition}\label{prop: crystal structure fpr fundamental paths}
For two vertices $v$ and $u$ in $V_j$ such that $u=f_i(v)$ for some $i\in [n]$, we have \begin{align}
    u=v-\alpha_i. \nonumber
\end{align} 
\end{proposition}

\begin{proof}
    By Definition~\ref{def: crystal operator on fundamental}, the action of any operator $f_i$ on any fundamental one-skeleton path $p$ is the simple reflection $s_i$ on $p$ if $f_i(p)\neq 0$. For any $j\in [n]$, the vertex set $V_j$ is just the orbit of $\omega _j$ under the action of the Weyl group $W$. Using the standard geometric realization of the root system of type $A_n$, we have 
    \begin{align}\label{eq: explicit formula for A_n crystal action}
        V_j=W\cdot \omega_j=\{ \Sigma_{s\in S}(\omega_s-\omega_{s-1}) \mid S\subseteq \{1,\cdots,n+1\}, |S|=j\}, 
    \end{align}
    where we set $\omega_0=\omega_{n+1}=0$.\par
    As a result, for $u=f_i(v)$ in $V_j$, the coefficient of $w_i$ in $v$ has to be 1. We have \begin{align}
        u=s_i(v)= v-\langle v,\alpha _i ^{\vee}\rangle \cdot \alpha _i=v-\langle \omega_i,\alpha _i ^{\vee}\rangle \cdot \alpha _i=v-\alpha_i. \nonumber
    \end{align}
\end{proof}
  By Equation~\eqref{eq: explicit formula for A_n crystal action}, any fundamental one-skeleton path $p$ is of the form $p=\Sigma_{i\in [n]}a_{pi}\omega_i$, where $a_{pi}\in\{-1,0,1\}$. For any $i\in [n]$, there is an $i$-colored arrow starting from $p$ if and only if $a_{pi}=1$. Symmetrically, there is an $i$-colored arrow ending in $v$ if and only if $a_{pi}=-1$. \par
  We now define a $G$-crystal structure on $V_j$. For any $v\in V_j$, we set $\varepsilon_i(v)=max\{k\in \mathbb{N} \mid e_i^{k}(v)\neq 0\}$ and $\varphi_i(v)=max\{k\in \mathbb{N} \mid f_i^{k}(v)\neq 0\}$. Then Proposition 2.1 induces the following corollary:
\begin{corollary}\label{coro: G-crystal}
    $(V_j,\wt, e_i, f_i, \varepsilon _i,\varphi _i)$ is a finite $G$-crystal for any $j\in [n]$.
\end{corollary}

Our next goal is to show that any one-skeleton path is a concatenation of fundamental one-skeleton paths. We consider the action of the Weyl group $W$ on $X$. The fundamental alcove $\Delta$ is the fundamental domain of this action. In this paper, we choose the standard orientation of 1-dimensional faces of $\Delta$:  \begin{align}
    E=\{\omega_i,\omega_j-\omega_k\}_{1\leq i\leq n,1\leq j<k\leq n}. \nonumber
\end{align}
There are $\frac{(n+1)n}{2}$ edges of $\Delta$. We have the following corollary of Proposition~\ref{prop: crystal structure fpr fundamental paths}:

\begin{proposition}\label{prop:inclusion of edges}
    We have the inclusion $E\subseteq P_f$.
\end{proposition}

\begin{proof}
Since $P_f=\bigcup_{j\in [n]}V_j\cup \{0\}$, we have $\omega_i\in V_i\subset P_f$ for any $i\in [n]$. By Equation~\ref{eq: explicit formula for A_n crystal action}, we have 
\begin{align}\label{eq: w_i-w_k formula}
    \omega_j-\omega_k=\Sigma_{s\in \{1,\cdots ,j\}\cup\{k+1,\cdots , n+1\}}(\omega_s-\omega_{s-1})\in V_{n-k+j+1},
\end{align}
for any $1\leq j<k \leq n$. Thus we have the inclusion 
\begin{align}
    E \subseteq P_f. \nonumber
\end{align}
\end{proof}

\begin{remark}\label{remark minus symmetric}
   We have the observation that $-\omega_i=\Sigma_{s\in \{i+1,\cdots , n+1\}}(\omega_s-\omega_{s-1})\in V_{n-i+1}$ for all $i\in [n]$. Similarly, we have $-E\subset P_f$ using Equation~\eqref{eq: w_i-w_k formula}.
\end{remark}

Now we are ready to obtain all one-skeleton paths from fundamental one-skeleton paths.

\begin{theorem}\label{thm: prime decomposition}
    Any nonzero one-skeleton path $p$ is of the form \begin{align}\label{eq: preime decomposition}
       p=p_1*\cdots * p_m, 
    \end{align}
     where $m\geq 1$ and $p_i\in P_f$ for any $1\leq i\leq m$. Equation~\eqref{eq: preime decomposition} is called the \textbf{prime decomposition} of $p$. Each $p_i$ is called a \textbf{factor} of $p$ and $m$ is the \textbf{length} of $p$.
\end{theorem}

\begin{proof}
    By Definition~\ref{def: one-skeleton path}, any one-skeleton path is piecewise an edge of an alcove. It suffices to prove that any edge of an alcove is a fundamental one-skeleton path up to translation by $\Lambda$. Since the Weyl group $W$ acts transitively on the set of alcoves containing $0$, and the symmetry between $E$ and $-E$, it suffices to prove $W\cdot E\cup W\cdot(-E)\subset P_f$, where $E=\{\omega_i,\omega_j-\omega_k\}_{1\leq i\leq n,1\leq j<k\leq n}$. \par
    
    Proposition~\ref{prop:inclusion of edges} implies that any $v\in E$ is in $V_j$ for some $j\in [n]$. Since $V_j$ is the orbit $W\cdot u$ for any $u\in V_j$, we have $W\cdot v\in P_f$. Taking the union $W\cdot E=\bigcup_{v\in E} W\cdot v$ , we have $W\cdot E \subset P_f$. Using Remark~\ref{remark minus symmetric}, we have $W\cdot (-E) \subset P_f$. As a result, any one-skeleton path is a concatenation of fundamental one-skeleton paths.
\end{proof}

Up to the equivalence relation defined in Remark~\ref{rem:equivalent paths}, the prime decomposition is unique for any one-skeleton path. Now we can define operators $f_i$ on the set of one-skeleton paths $OP$.
\begin{definition}\label{def:crystal operation on general paths}
    For a one-skeleton path $p$ with the prime decomposition $p=p_1*\cdots * p_m$, the operators $\{f_i\}_{i\in [n]}$ are defined as follows:
     \begin{equation}
        f_i(p) =
\begin{cases}
p_1*\cdots f_i(p_j)*\cdots p_m & \textit{j is the minimal index such that}\ \langle p_j(t),\alpha_i^\vee\rangle>0, t\in (0,1]\\
0  & \textit{if} \ \langle p_j(t),\alpha_i^\vee\rangle\leq0 \ \textit{for all }\ 1\leq j\leq m, t\in [0,1].
\end{cases} \nonumber
    \end{equation}
\end{definition}
We cannot define a crystal structure on $OP$ using $f_i$ directly since $f_i$ is not injective on the preimage set $f_i^{-1}(OP)$. This motivates the polytopal map introduced in the next subsection, which gives MV polytopes with crystal structures for certain one-skeleton paths.

\begin{remark}\label{rem:littelmann operators}
   The operators $\{f_i\}_{i\in [n]}$ on fundamental one-skeleton paths defined here coincide with the root operators defined in~\ref{def: crystal operator on fundamental}. For any fundamental one-skeleton $p$ and any simple root $\alpha_i$, the function $t\to \langle p(t), \alpha_i^\vee\rangle$ is monotone for $t\in [0,1]$. By Equation~\eqref{eq: explicit formula for A_n crystal action}, we have $|\langle p(t), \alpha_i^\vee\rangle|\leq 1$ for any fundamental one-skeleton path $p$ and any simple root $\alpha_i$ is. Thus the operators $f_i$ coincide with the root operators defined by Littelmann in Section 1 in \cite{LP}. For a general one-skeleton path, our operators $f_i$ are different from Littelmann's root operators. In particular, we cannot define inverse operators $e_i$ for a general one-skeleton path.
\end{remark}

\subsection{Polytopal map}\label{subsection: polytopal map}

We associate a polytope to any one-skeleton path using the operators $\{f_i\}_{i\in [n]}$.

\begin{definition}\label{def: polytopal map}
    For any nonzero one-skeleton path $p$, the \textbf{integral polytope} associated with $p$ is \[
    \Pol(p)=\Conv\{\wt(w(p))\}_{w\in \W, w(p)\neq0}.
    \]
\end{definition}

The map $\Pol$ is from the set of one-skeleton paths $OP$ to the set of integral polytopes in the weight space, which is called the \textbf{polytopal map}. The following theorem shows that the polytopal map sends the concatenation of one-skeleton paths to the Minkowski sum of polytopes.  

\begin{theorem}\label{thm: pol and concatenation are compatible}
    For $p,q\in OP$, we have\begin{align}
        \Pol(p*q)=\Pol(p)\boxplus \Pol(q). \nonumber
    \end{align} 
\end{theorem}

 We need the following two lemmas to prove Theorem~\ref{thm: pol and concatenation are compatible}:
 \begin{lemma}\label{lem: left inclusion}
     For $p,q\in OP$, we have
    \begin{align}
        \wt(w_1(p)*w_2(q))=\wt(w_2\cdot w_1(p*q))  \nonumber
    \end{align}
for any $w_1,w_2\in \W$ such that $w_1(p)$ and $w_2(q)$ are not $0$.
 \end{lemma}

\begin{proof}[Proof of Lemma~\ref{lem: left inclusion}]
  We write the prime decomposition of $p$ and $q$ as follows:\begin{align}
      p=p_1*\cdots*p_{m}\ \textit{and }
      q=p_{m+1}*\cdots *p_{m+m'}. \nonumber
  \end{align}
  
  Suppose that the word decompositions of $w_1$ and $w_2$ are:
  \begin{align}
      w_1=f_{j_{1,1}}\cdots f_{j_{1,s_1}} \ \textit{and } w_2=f_{j_{2,1}}\cdots f_{j_{2,s_2}}. \nonumber
  \end{align}
  
  Notice that $w_k(r)\neq 0$ implies that $\wt(w_k(r))=\wt(r)-(\Sigma_{1\leq i\leq s_k}\alpha_{k,i})$ for any $k\in\{1,2\}$ and $r\in OP$, where $\alpha_{k,i}$ is the simple root corresponding to $s_{j_{k,i}}$.\par
  
  Since $w_1(p)\neq 0$, the action of any letter $f_{j_{1,i}}$ in $w_1$ on $p*q$ changes one of fundamental one-skeleton path among $\{p_1,\cdots , p_{m}\}$ for all $1\leq i \leq s_1$. We have $w_1(p*q)=w_1(p)*q$.  For any $1\leq i\leq s_2$, there exists at least one fundamental one-skeleton $p^+_i$ in the last $m'$ factors of $w_1(p)*q$ such that $\langle p^+_i(1),\alpha^{\vee}_{j_{2,i}}\rangle >0$, as a consequence of the condition $w_2(q)\neq 0$. We have $w_2(w_1(p)*q)\neq 0$.\par
  
  As a result, we have \begin{align}
      \wt(w_2\cdot w_1(p*q))&=\wt(w_2(w_1(p)*q))=\wt(w_1(p)*q)-(\Sigma_{1\leq i\leq s_2}\alpha_{2,i}) \nonumber \\
      &=\wt(p*q)-(\Sigma_{1\leq i\leq s_1}\alpha_{1,i})-(\Sigma_{1\leq i\leq s_2}\alpha_{2,i}) \nonumber \\
      &=\wt(p)+\wt(q)-(\Sigma_{1\leq i\leq s_1}\alpha_{1,i})-(\Sigma_{1\leq i\leq s_2}\alpha_{2,i}) \nonumber \\
      &=\wt(p)-(\Sigma_{1\leq i\leq s_1}\alpha_{1,i})+\wt(q)-(\Sigma_{1\leq i\leq s_2}\alpha_{2,i}) \nonumber \\
      &=\wt(w_1(p))+\wt(w_2(q)) \nonumber \\
      &=\wt(w_1(p)*w_2(q)). \nonumber
  \end{align}
\end{proof}

\begin{lemma}\label{lem: right inclusion}
    For any $p,q\in OP$ and $w\in \W$ such that $w(p*q)\neq 0$, there exist $w_1,w_2\in \mathcal{W}$ such that \begin{align}
        \wt(w(p*q))=\wt(w_1(p)*w_2(q)). \nonumber
    \end{align} 
\end{lemma}

\begin{proof}[Proof of Lemma~\ref{lem: right inclusion}]
 We write the prime decomposition of $p$ and $q$ as follows:\begin{align}
      p=p_1*\cdots*p_{m}\ \textit{and }
      q=p_{m+1}*\cdots *q_{m+m'}. \nonumber
  \end{align}
  Suppose that the word decomposition of $w$ is:
  \begin{align}
      w=f_{j_1}\cdots f_{j_s}. \nonumber
  \end{align}
  
  We define a map $c$ from $\{1,\cdots ,s\}$ to $\{1,2\}$ by backward induction from $s$ to $1$ as follows:\begin{align}
      c(i)=
      \begin{cases}
          1 & \textit{if } f_{j_i} \textit{ changes one of the first }m \textit{ factors of } f_{j_{i+1}}\cdots f_{j_s}(p*q),\\
          2 & \textit{if } f_{j_i} \textit{ changes one of the last } m'\textit{ factors of } f_{j_{i+1}}\cdots f_{j_s}(p*q).
      \end{cases} \nonumber
  \end{align}
  Because of the assumption $w(p*q)\neq 0$, we have a partition of the set $\{1,\cdots s\}=c^{-1}(1)\sqcup c^{-1}(2)$, where $c^{-1}(1)=\{k_1<\cdots <k_l\}$ and $c^{-1}(2)=\{k_{l+1}<\cdots <k_s\}$. Then we have
  \begin{align}
      \wt(w(p*q))=\wt(w_1(p)*w_2(q)), \nonumber
  \end{align}
  where $w_1=f_{j_{k_1}}\cdots f_{j_{k_l}}$ and $w_2=f_{j_{k_{l+1}}}\cdots f_{j_{k_s}}$.
\end{proof}

\begin{proof}[Proof of Theorem~\ref{thm: pol and concatenation are compatible}]
    Given two one-skeleton paths $p$ and $q$, we have \begin{align}
       \Pol(p*q) &=\Conv(\wt(w(p*q)))_{w\in \W, w(p*q)\neq 0} \nonumber \\
       & \subseteq \Conv(\wt(w_1(p)*w_2(q)))_{w_1,w_2\in \W, w_1(p), w_2(q)\neq 0} \nonumber \\
       &= \Conv(\wt(w_1(p))+\wt(w_2(q)))_{w_1,w_2\in \W, w_1(p), w_2(q)\neq 0} \nonumber \\
       &=\Conv(\wt(w_1(p)))_{w_1\in \W, w_1(p)\neq 0}\boxplus \Conv(\wt(w_2(q)))_{w_2\in \W, w_2(q)\neq 0}  \nonumber \\
       &=\Pol(p)\boxplus \Pol(q), \nonumber
    \end{align}
    by Lemma~\ref{lem: right inclusion}.\par
    
    Conversely, we obtain \begin{align}
       \Pol(p)\boxplus \Pol(q) &=\Conv(\wt(w_1(p)))_{w_1\in \W, w_1(p)\neq 0}\boxplus \Conv(\wt(w_2(q)))_{w_2\in \W, w_2(q)\neq 0} \nonumber \\
       &= \Conv(\wt(w_1(p))+\wt(w_2(q)))_{w_1,w_2\in \W, w_1(p), w_2(q)\neq 0}\nonumber \\
       &= \Conv(\wt(w_1(p)*w_2(q)))_{w_1,w_2\in \W, w_1(p), w_2(q)\neq 0} \nonumber \\
       &= \Conv(\wt(w_2 \cdot w_1(p*q)))_{w_1,w_2\in \W, w_1(p), w_2(q)\neq 0} \nonumber \\
       & \subseteq \Conv(\wt(w(p*q)))_{w\in \W, w(p*q)\neq 0}  \nonumber \\
       & =\Pol(p*q), \nonumber
    \end{align}
    by Lemma~\ref{lem: left inclusion}.\par
    
    Hence, we have \begin{align}
       \Pol(p*q)=\Pol(p)\boxplus \Pol(q). \nonumber
    \end{align}
\end{proof}

A corollary of Theorem~\ref{thm: pol and concatenation are compatible} is that the order of the concatenation of one-skeleton paths does not affect the associated polytope.
\begin{corollary}\label{coro: commute}
    For any $p,q\in OP$, we have \begin{align}
        \Pol(p*q)=\Pol(q*p). \nonumber
    \end{align}
\end{corollary}

\begin{proof}
    Using Theorem~\ref{thm: pol and concatenation are compatible} and the commutativity of the Minkowski sum, we have \begin{align}
    \Pol(p*q)=\Pol(p)\boxplus \Pol(q)=\Pol(q)\boxplus \Pol(p)=\Pol(q*p). \nonumber
    \end{align}
\end{proof}

\subsection{The free path algebra \texorpdfstring{$\Fp$}{Fp}}\label{subsection: free path algebra}

\begin{definition}\label{def: free path algebra}
     The \textbf{free path algebra} $\mathcal{F}_p$ of $G$ is the free commutative
     $\mathbb{C}$-algebra generated by the set $P_f$ of fundamental one-skeleton paths.
     Equivalently,
     \[
        \mathcal{F}_p := \operatorname{Sym}_{\mathbb{C}}(\mathbb{C}P_f),
     \]
     where $\mathbb{C}P_f$ denotes the $\mathbb{C}$-vector space with basis $P_f$.
\end{definition}

A monomial $p_1\cdots p_m$ in $\mathcal{F}_p$ can be represented geometrically by the
concatenated one-skeleton path
\[
    p_1 * \cdots * p_m.
\]
Although the concatenation of paths depends on the order of the factors, Corollary~\ref{coro: commute}
shows that the associated polytope
\[
    \operatorname{Pol}(p_1 * \cdots * p_m)
\]
is independent of this order. Hence this geometric representative is compatible with the
commutative algebra structure of $\mathcal{F}_p$ after applying the polytopal map.

\begin{theorem}\label{thm: fundamental MV polytopes}
    For any fundamental one-skeleton path $p\in P_f$, the polytope $\Pol(p)$ is an MV polytope. 
\end{theorem}

\begin{proof}
 By Remark~\ref{rem:littelmann operators}, the operators $\{f_i\}_{i\in [n]}$ on a fundamental one-skeleton path coincide with the root operators $f_{\alpha_i}$ defined by Littelmann for paths in $X$ in \cite{LP}.\par
   Theorem 4.6 and Theorem 4.9 in \cite{BG} equipped with Theorem 8.1 in \cite{ehrig} imply that the convex hull of all endpoints of images of $p$ under Littelmann's root operators $f_{\alpha_i}$ is an MV polytope. By the consistency between the operators $f_i$ and Littelmann's operators $f_{\alpha_i}$ established above, we have that $\Pol(p)$ is an MV polytope for any fundamental one-skeleton path $p$.   
\end{proof}

Given a perfect basis $\mathcal{B}=\{b_i\}_{i\in B(\infty)}$ of $\CN$, we have a unique $G$-crystal isomorphism $\varphi_{\mathcal{B}}$ from $\MV$ to $\B$ defined by $\varphi _{\B}(P_i)=b_i$ for any $i\in B(\infty)$, which preserves the $G$-crystal structure. Any MV polytope is the Minkowski sum of prime MV polytopes. Theorem~\ref{thm: fundamental MV polytopes} lets us deduce a family of MV polytopes from one-skeleton paths, which induces an algebra homomorphism from $\Fp$ to $\CN$ with respect to $\mathcal{B}$. 

\begin{definition}\label{def: algebraic one-skeleton map}
Given a perfect basis \(\mathcal{B}=\{b_i\}_{i\in B(\infty)}\) of \(\mathbb C[N]\), Theorem~\ref{thm: fundamental MV polytopes} gives $\varphi_{\mathcal{B}}(\operatorname{Pol}(p))\in \mathcal{B}
$
for every \(p\in P_f\). Hence the assignment
$p\longmapsto \varphi_{\mathcal{B}}(\operatorname{Pol}(p))$, $ p\in P_f$,
extends uniquely to an algebra homomorphism
\[
    \phi_{\mathcal{B}}:\mathcal F_p\longrightarrow \mathbb C[N].
\]
We call \(\phi_{\mathcal{B}}\) the \textbf{one-skeleton algebraic map} with respect to \(\mathcal{B}\).
\end{definition}
The one-skeleton algebraic map with respect to the unique perfect basis of $\CN$ is surjective when $G$ is of type $A_2$ or $A_3$. Moreover, we prove that the one-skeleton algebraic map with respect to the semicanonical basis is always surjective for any $G$ of type $A_n$ in Section~\ref{sec:mvpolytopes}.\par

In general, suppose that \(p=p_1*\cdots *p_m\) is a one-skeleton path whose factors give
MV polytopes with the same BZ-choice in the sense of Section 6.1 of \cite{JOEL}. Then
\(\operatorname{Pol}(p)\) is again an MV polytope, and
\[
\phi_{\mathcal{B}}(p_1\cdots p_m)=\varphi_{\mathcal{B}}(\operatorname{Pol}(p)).
\]
Here we use the result that the set of MV polytopes with the same BZ-choice is closed under the Minkowski sum by Theorem 6.2 in \cite{JOEL}.
\begin{remark}\label{rem: multiplication property}
    A perfect basis $\B$ is said to satisfy the multiplication property if \begin{align}
        \varphi_{\B}(P_i\boxplus P_j)=\varphi _{\B}(P_i)\cdot \varphi _{\B}(P_j), \nonumber
    \end{align}
    for any $i,j\in B(\infty)$ such that the Minkowski sum $P_i\boxplus P_j$ is still an MV polytope. The unique perfect basis of $\CN$ satisfies the multiplication property if $G$ is of type $A_2$.
\end{remark}

\section{MV polytopes from fundamental one-skeleton paths}\label{sec:mvpolytopes}
In this section, we prove that the MV polytope induced by a fundamental one-skeleton path is determined by its crystal graph. We show that the MV polytope induced by a fundamental one-skeleton path is just the Harder--Narasimhan polytope of the preprojective module giving the same polynomial in the semicanonical basis of $\CN$.

\subsection{Crystal graphs for fundamental one-skeleton paths}

By Definition~\ref{def: j-chain}, we can define a preorder $\prec$ on the vertex set $V_j$ of any $j$-chain by letting $y\prec x$ if and only if there is a path from $x$ to $y$ in the $j$-chain $Q_j$. Using Equation~\eqref{eq: w_i-w_k formula} in the proof of Proposition~\ref{prop:inclusion of edges}, $V_j$ has a unique minimal element $-\omega_{n+1-j}$ and a unique maximal element $\omega_j$ under the preorder $\prec$. \par

For any $p\in V_j$, we consider the full subquiver $Q_p$ of $Q_j$ whose vertex set consists of $p$, $-\omega_{n+1-j}$ and all elements $x\in V_j$ such that $-\omega_{n+1-j}\prec x \prec p$. The proof of Theorem~\ref{thm: fundamental MV polytopes} shows that the convex hull of all vertices of $Q_p$ is an MV polytope. \par
 
Given an MV polytope $P$, we use the notation $B_P$ to denote the set \begin{align}
    \{x \in \MV \mid x=f'_{i_1}\cdots f'_{i_m}(P), i_k\in[n]\text{ for all }k\in [m] \}, \nonumber
\end{align}
which has a $G$-crystal structure equipped with the highest weight element $P$.\par

\begin{definition}\label{crystal graph}
    Given a $G$-crystal $B$, the \textbf{crystal graph} associated with $B$ is an $n$-colored quiver with vertex set $B$. There is an $i$-arrow form $b$ to $b'$ if and only if $b'=f_i(b)$ in $B$. 
\end{definition}

Now if we identify any vertex $q$ of $Q_p$ with the MV polytope $\Pol(q)$ in $\MV$, we have the following proposition induced by Theorem~\ref{thm: fundamental MV polytopes}:
\begin{proposition}\label{prop: crystal graph}
    $Q_p$ is the crystal graph associated with $B_{\Pol(p)}$ for any fundamental one-skeleton path $p$.
\end{proposition}

\begin{proof}
    Suppose we have a fundamental one-skeleton path $p=\Sigma_{i\in [n]}a_i \omega_i$ in the $j$-chain $Q_j$ for some $j\in [n]$. \par
    We claim that \begin{align}\label{eq: crystal graph and mv polytopes}
        \Pol(f_i(p))=f'_i(\Pol(p)), 
    \end{align}
    for any $i\in [n]$. Here, we use the notation $f'_i$ to represent the crystal operators acting on $\MV$. By the definition of the crystal operators on MV polytopes in \cite{CRMV}, we have $f'_i(\Pol(p))\neq 0$ if and only if $p-\alpha_i\in \Pol(p)$. By Proposition~\ref{prop: crystal structure fpr fundamental paths}, the weight $p-\alpha_i$ is in $Pol(p)$ if and only if $a_i=1$. In this case, $\wt(f'_i(\Pol(p)))=p-\alpha_i=f_i(p)\in \Pol(f_i(p))$. \par
    
    Any path in $Q_p$ starting from $p$ to $-\omega_{n+1-j}$ gives the $\textbf{i}$-Lusztig datum for $\Pol(p)$, where $\mathbf{i}$ is a certain reduced expression of the longest word $w_0$ in the Weyl group $W$. We choose the path $l$ starting from $p$, traversing $f_i(p)$ and ending in $-\omega_{n+1-j}$. We cancel the first arrow $f_j$ in $l$ to obtain a path $l'$ from $f_i(p)$ to $-\omega_{n+1-j}$, which is the $\mathbf{i}$-Lusztig datum for $\Pol(f_i(p))$. Hence the $\textbf{i}$-Lusztig datum for $\Pol(p)$ and $\Pol(f_i(p))$ just differ by 1 in the first nonzero $i$-position of the $\mathbf{i}$-Lusztig datum for $\Pol(p)$.\par
    
    On the other hand, the $\mathbf{i}-$Lusztig datum $f'_i(\Pol(p))$  is given by just subtracting 1 in the first nonzero $i$-position of the $\mathbf{i}$-Lusztig datum for $\Pol(p)$ using Proposition 3.4 in \cite{CRMV}. Therefore we have $\Pol(f_i(p))=f'_i(\Pol(p))$ since they have the same $\mathbf{i}$-Lusztig datum by Theorem~\ref{mvpolytope}. \par
    Equation~\eqref{eq: crystal graph and mv polytopes} implies the isomorphism of quivers between $Q_p$ and the crystal graph of $\Pol(p)$ via the polytopal map. 
\end{proof}

 Proposition~\ref{prop: crystal graph} implies that the operators $e_i$ and $f_i$ defined in Definition~\ref{def: crystal operator on fundamental} and Definition~\ref{def: inverse fundamental operator} are compatible with the crystal operators on $\MV$. Without ambiguity, we also use $\W$ to represent the monoid of words consisting of crystal operators $f'_i$ on the $G$-crystal $\MV$.\par

An MV polytope $P$ is uniquely determined by its vertex set $Ver(P)$. We have $\wt(f_iP)=\wt(P)-\alpha _i$ if $f'_iP\neq 0$. Given a dominant weight $\lambda\in \Lambda_+$, the irreducible representation $V(\lambda)$ of $G$ gives an MV polytope $P(\lambda)$. Using the crystal embedding from $B_{P(\lambda)}$ to $B(\infty)$ induced from Proposition 2.7 in \cite{MVHD}, we have the following proposition:

\begin{proposition}[Proposition 2.7 in \cite{MVHD}]\label{prop:vertex of P}
    For any dominant weight $\lambda\in \Lambda_+$, we have \begin{equation}
        Ver(P(\lambda))=\{\wt(w\cdot P(\lambda))\mid w\in \W\}.
    \end{equation}
\end{proposition}

We say that an MV polytope $P$ is an \textbf{MV subpolytope} of another MV polytope $P'$ if $P=w\cdot P'$ for some $w\in \W$. Any MV polytope is an MV subpolytope of $P(\lambda)$ for some $\lambda\in X_+$. \par

\begin{definition}\label{def: good MV}
    An MV polytope $P$ is \textbf{good} if \begin{align}
        Ver(P)\subseteq \{\wt(x)\}_{x\in B_P}. \nonumber
    \end{align}
\end{definition}

We have the following example, which is a corollary of Proposition~\ref{prop:vertex of P}:
\begin{example}\label{example: good P}
     For any dominant weight $\lambda\in X_+$, the MV polytope $P(\lambda)$ is good.
\end{example}

\begin{example}\label{example: bad P}
    We consider the following MV polytope $P$:
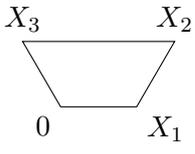
\par
    \begin{tikzpicture}
\draw (-1,0) node[above] {$X_3$}
  -- (1,0) node[above] {$X_2$}
  -- ({1/2},{-sqrt(3)/2}) node[below right] {$X_1$}
  -- ({-1/2},{-sqrt(3)/2}) node[below left] {$0$}
  -- cycle;
\end{tikzpicture},\par
where $X_1=\alpha_1$, $X_2=2\alpha_1+\alpha_2$ and $X_3=\alpha_2$. This is the MV polytope of type $A_2$. The set $\{\wt(x)\}_{x\in B_P}$ is \begin{align}
    \{0, \alpha_2, \alpha_1+\alpha_2, 2\alpha_1+\alpha_2\}. \nonumber
\end{align}
The vertex $X_1$ of $P$ is not in this set. Thus $P$ is not good.
\end{example}

We have the following corollary of Proposition~\ref{prop: crystal graph}: 
\begin{corollary}\label{coro: fundamental mv are good}
    For any fundamental one-skeleton path $p$, the MV polytope $\Pol(p)$ is good.
\end{corollary}

\begin{proof}
   For any fundamental one-skeleton path $p$, the vertex set of the MV polytope $\Pol(p)$ is contained in the vertex set of $Q_p$ by Definition~\ref{crystal graph}. By Proposition~\ref{prop: crystal graph}, the vertex set of $Q_p$ is contained in the set $\{\wt(x)\}_{x\in B_{\Pol(p)}}$. 
\end{proof}

\subsection{1-filtered indecomposable modules}\label{subscetion: 1-filtered indecompsable}

We denote the preprojective algebra associated with $SL_{n+1}(\mathbb{C})$ by $\Lambda_{n}$ \cite{proj1}. The irreducible components of the representation variety of $\Lambda_n$ have a $G$-crystal structure $B(\infty)$ by results in \cite{LUSZSEMI}. A more explicit description is given in \cite{prpo}. We recall the basic definitions of the preprojective algebra in \cite{prpo} and \cite{proj1}.

For $n\ge 2$, let $\tau_n$ be the Dynkin quiver 
\[
\xymatrix{
1 & 2 \ar[l]_{\alpha_1} & \cdots \ar[l]_{\alpha_2} & n \ar[l]_{\alpha_{n-1}}}
\]
of type $A_n$.
Thus
$
\Lambda_n = \mathbb{C} \overline{\tau}_n/J_n
$
where the double quiver $\overline{\tau}_n$ of $\tau_n$ is
\[
\xymatrix{
1\ar@<1ex>[r]^{\alpha^*_1}&\ar@<1ex>[l]^{\alpha_1} 2 
\ar@<1ex>[r]^{\alpha^*_2}&\ar@<1ex>[l]^{\alpha_2} 
\ar@<1ex>[r]^{\alpha^*_{n-1}} 
\cdots & n \ar@<1ex>[l]^{\alpha_{n-1}},
}
\]
and the ideal $J_n$ is generated by  
\begin{align}
\alpha_1\alpha^*_1,\quad \alpha^*_{n-1}\alpha_{n-1},\quad
\alpha^*_i\alpha_i - \alpha_{i+1}\alpha^*_{i+1},\quad (1 \leq i \leq n-2). \nonumber
\end{align}

Any $\Lambda_n$-module $M$ gives a dimension vector $\underline{\dim}M$. For each vertex $i\in [n]$, denote by $M_i$ the vector space assigned to $i$.  
The \emph{dimension vector} of $M$ is defined by the element
\[
    \underline{\dim}M := (\dim M_i)_{i\in I} \in \mathbb{N}^{n}.
\]
We consider the representation variety \[
mod(\Lambda_n,m)=\{M\mid M \text{ is a } \Lambda_n \text{-module and }\underline{\dim}M=m\},
\]
for any $m\in \mathbb{N}^n$. The set of irreducible components of $Rep(\Lambda_n, m)$ is denoted by $\Irr_{\Lambda_n}(m)$. The collection $\Ir=\bigsqcup_{m\in \mathbb{N}^n}\Irr_{\Lambda_n}(m)$ has the $G$-crystal structure
$B(\infty)$ by the following parameterization. \par
For any $\Lambda_n$-module $M$, we define its 
\textbf{Harder--Narasimhan (HN) polytope} to be
\[
  \operatorname{HN}(M) := \Conv\{\underline{\dim} N \mid N \subseteq M \text{ is a submodule}\}.
\]
If we identify $\mathbb{N}^n$ with $\mathbb{N}\alpha_1\oplus \cdots \oplus \mathbb{N}\alpha_n$, the HN polytope can be viewed as a polytope lying in $X$. 
\begin{theorem}[Theorem 11.3 in \cite{MVHD}]\label{thm: HN and MV}
There exists a bijection $\beta$ from $\Ir$ to $\MV$ satisfying the following condition: Let $Y\in \Ir$ be an irreducible component. For a general point $M$ in $Y$, $\operatorname{HN}(M)=\beta(Y)$. 
\end{theorem}

Now we show that MV polytopes given by fundamental one-skeleton paths are HN polytopes given by certain indecomposable modules of $\Lambda_n$. We use the Galois cover of $\tau_n$ to study a special family of indecomposable modules of $\Lambda_n$. Here we recall some notations about Galois covers in \cite{proj1}. \par
Let 
$
\overline{\Lambda}_n = \mathbb{C} \overline{\tau}_n/\widetilde{J}_n
$
where $\overline {\tau}_n$ is the quiver with vertices
$\{ i_j \mid 1 \leq i \leq n, j \in \mathbb{Z} \}$
and arrows 
\[
\alpha_{i,j}: (i+1)_j \to i_j,\quad \alpha^*_{i,j}: i_j \to (i+1)_{j-1},\quad
(1 \leq i \leq n-1, j \in \mathbb{Z}),
\]
and the ideal $\widetilde{J}_n$ is generated by
\[
  \alpha_{1,j}\alpha^*_{1,j+1},\quad \alpha^*_{n-1,j}\alpha_{n-1,j},\quad 
\alpha^*_{i,j}\alpha_{i,j} - \alpha_{i+1,j-1}\alpha^*_{i+1,j},\quad
(1 \leq i \leq n-2, j \in \mathbb{Z}).
\]

The group $\Z$ acts on $\overline{\Lambda}_n$ by $\C$-linear automorphisms via
\[
z \cdot i_j = i_{j+z},\qquad
z \cdot \alpha_{ij} = \alpha_{i,j+z},\qquad
z \cdot \alpha^*_{ij} = \alpha^*_{i,j+z}.
\]
This induces an action 
\[
\Z \times \md(\overline{\Lambda}_n) \longrightarrow \md(\Lambda_n)
\]
\[
(z,M) \mapsto {^{(z)}M},
\]
where ${^{(z)}M}$ denotes the $\overline{\Lambda}_n$-module obtained from
$M$ by translating $M$ $z$ levels upwards. 

If we consider $\overline{\Lambda}_n$ and $\Lambda_n$ as locally bounded categories 
we have a functor $\F: \overline{\Lambda}_n \longrightarrow \Lambda_n$
defined by 
\[
i_j \mapsto i,\qquad 
\alpha_{ij} \mapsto \alpha_i,\qquad 
\alpha^*_{ij} \mapsto \alpha^*_i.
\]
This is a Galois covering of $\Lambda_n$ with Galois group $\Z$ \cite{proj1}.
It provides us with the push-down functor
\[
\md(\overline{\Lambda}_n) \longrightarrow \md(\Lambda_n),
\]
which we also denote by $\F$. It is defined as follows:\par

Let $x \in \md(\overline{\Lambda}_n,\mathbf{V})$ be a $\overline{\Lambda}_n$-module with 
underlying graded vector space $\mathbf{V}=\bigoplus_{i,j}V_{i_j}$. 
Then $\F(x)$ has the same underlying vector space with the grading $\mathbf{V} = \bigoplus_i V_i$ where
$V_i=\bigoplus_j V_{i_j}$,
and $\F(x)$ has maps
$\F(x)_{\alpha_i} = \bigoplus_j\, x_{\alpha_{ij}}$ and
$\F(x)_{\alpha^*_i}=\bigoplus_j\, x_{\alpha^*_{ij}}$.\par

Now we can define 1-filtered indecomposable modules of $\Lambda_n$ as follows:

\begin{definition}\label{def: 1-filtered indecomposable}
    An indecomposable module $M$ of $\Lambda_n$ is \textbf{1-filtered} if $M=\F(\widetilde{M})$ for some indecomposable $\widetilde{M} \in \md(\overline{\Lambda}_n)$ satisfying $M_{i_j}\cong \C$ for all nonzero $V_{i_j}$, $M_{i,j}\cong \id_{\C}$ if $M_{i_j}\cong M_{(i+1)_{j}}\cong \C$ and $M^*_{i,j}\cong \id _{\C}$ if $M_{i_j}\cong M_{(i+1)_{j-1}}\cong \C$.
\end{definition}

\begin{example}\label{example: projective are 1-filtered}
    For any $k\in [n]$, the projective indecomposable module $P_k$ of $\Lambda_n$ is 1-filtered. We can construct a  $\overline{\Lambda}_n$-module $\widetilde{P}_k$ by setting 
    
    \[\dim(\widetilde{P}_k)_{i_j}=\text{number of equivalence classes of paths of length j from k to i in } Q_n, \]
    where the equivalence relations of paths are induced by generators of $J$. We have $\F(\widetilde{P}_k)=P_k$.
\end{example}

We consider the set of 1-filtered indecomposable modules that are submodules of projective indecomposable modules, which is denoted by $Ind_1(\Lambda_n)$.

\begin{proposition}\label{prop: fundamental are 1-filtered}
    For any fundamental one-skeleton path $p\in P_f$, we have \[\Pol(p)\in \{\operatorname{HN}(M) \mid M\in Ind_1(\Lambda_n)\}.\]
\end{proposition}

\begin{proof}
We consider the composition series of $P_k$. The composition series of $P_k$ is induced from the composition series of $\widetilde{P}_k$ via the functor $F$. We define an  $n$-colored quiver $R_k$ whose vertices are submodules of $P_k$. There is an $i$-colored arrow from $M$ to $M'$ if and only if the quotient $M/M'$ is the simple module $S_i$ for any $i\in [n]$. By Example~\ref{example: projective are 1-filtered}, we have $R_k$ is isomorphic to the $k$-chain $Q_k$ as an $n$-colored quiver. Thus we have $\Pol(\omega_i)=\operatorname{HN}(P_i)$ for any $i\in [n]$. \par
Given a $1$-filtered indecomposable module $M\in Ind_1(\Lambda_n)$, we suppose that $M$ is a submodule of $P_k$ for some $k\in [n]$. There exist $j_1,\cdots ,j_m$ in $[n]$ and $\Pi_n$-modules $M_{j_i}\in Ind_1(\Lambda_n)$ such that $P_k/M_{j_i}\cong S_{j_i}$ for any $i\in[m]$ and $M_{j_m}\cong M$. We have $\operatorname{HN}(M)=\Pol(f_{j_1}\cdots f_{j_m}\omega_k)$ using the quiver isomorphism between $R_k$ and $Q_k$.
\end{proof}

\section{Fundamental one-skeleton paths and initial cluster variables of \texorpdfstring{$\CN$}{CN}}\label{sec: cluster variables}

In this section, we construct the correspondence between cluster variables in the initial seed of $\CN$ and oriented edges of the fundamental alcove via MV polytopes. Using one-skeleton paths, we have a geometric interpretation of initial cluster variables and certain mutation relations.

\subsection{Initial seed of the cluster algebra \texorpdfstring{$\CN$}{CN}}

We recall results in the work by Geiss, Leclerc, and Schr\"{o}er on the cluster algebra structure of $\CN$. Suppose that $G$ is of type $A_n$. The longest word $w_0$ in the Weyl group $W$ admits an expression $\mathbf{i}=(i_{j})_{1\leq j\leq \frac{(n+1)n}{2}}=(1,\cdots,n,1,\cdots,n-1,\cdots,1,2,1)$.\par

We add $n$ additional letters $i_{-n}, \ldots, i_{-1}$ at the beginning of $\mathbf{i}$, where $i_{-j} = -j$, and obtain an $(\rn+n)$-tuple
\[
\bar{\mathbf{i}}=(i_{-n}, \ldots, i_{-1}, i_{1}, \ldots, i_{\rn}) = (-n, \ldots, -1, 1, \cdots, n,\cdots,1).
\]
For $k\in [-n, -1] \cup [1, \rn]$, let
\[
k^+ = \begin{cases}
    \rn+1 & \text{if } |i_{l}| \neq |i_{k}| \text{ for all } l > k, \\
    \min\{l \mid l > k \text{ and } |i_l| = |i_k|\} & \text{otherwise}.
\end{cases}
\]

Then the index $k$ is called $\mathbf{i}$-\textit{exchangeable} if $k$ and $k^+$ are in $[1, \rn]$. Let $e(\In) \subset [1, \rn]$ be the set of $\In$-exchangeable elements. One easily checks that $e(\In)$ contains $\rn -n$ elements. In fact, $k$ is $\In$-exchangeable if and only if $k\notin \{-n,\cdots,-1\}\cup\{\frac{(n+l)(n-l+1)}{2}\}_{1\leq l\leq n}$.

Next, one defines the exchange quiver $\tilde{Q_{\In}}$ with set of vertices $[-n, -1] \cup [1, \rn]$. Assume that $k$ and $l$ are indices such that the following two conditions hold:
\begin{itemize}
    \item $k < l$;
    \item $\{k, l\} \cap e(\In) \neq \emptyset$.
\end{itemize}

There is an arrow $k \to l$ in $\tilde{Q_{\In}}$ if and only if $k^+ = l$, and there is an arrow $l \to k$ if and only if $l < k^+ < l^+$ and $a_{|i_{k}|, |i_{l}|} = -1$. Here, $(a_{i,j})_{1 \leq i,j \leq n}$ denotes the Cartan matrix of $W$. We denote the exchange matrix associated with the quiver $\tilde{Q_{\In}}$ by $\tilde{A_{\In}}$. The following theorem recalls results about the initial seed in \cite{initial1}.

\begin{theorem}[Section 4 \cite{initial1}]\label{thm: initial seed functions}
We have
\begin{align}
\{ (\Delta(k, \In))_{k \in \bar{\In}}, \tilde{A_{\In}}\} \nonumber
\end{align}
is the initial seed for $\mathbb{C}[B \cap B_{-}\bar{w}_0B_-]$. The mutable variables are those $\Delta(k, \In)$ with $k \in e(\In)$. Here $\bar{w}_0$ is a lift of $w_0$ in $G$ and $\Delta(k, \In)$ are the generalized minors.\par

Moreover, if we remove the variables in the initial seed corresponding to the $n$ non-$\In$-exchangeable elements in $\{1, \ldots, \rn\}$, remove the corresponding vertices in $\tilde{A}_{\In}$, and restrict the generalized minors $\Delta(k,\In)$ to $N$, then we obtain an initial seed for $\mathbb{C}[N]$. In particular, each cluster has $\rn$ cluster variables, $n$ of them being frozen.
\end{theorem}
Before the restriction to \(N\), the seed for
\(\mathbb C[B\cap B^-\bar{w}_0B^-]\) has \(2n\) frozen variables, indexed by
the set 
$[-n,-1]\cup\{k\in[1,\frac{n(n+1)}{2}]\mid k^+=\frac{n(n+1)}{2}+1\}.$
After restriction to \(N\), the variables indexed by \(\{k\in[1,\frac{n(n+1)}{2}]\mid k^+=\frac{n(n+1)}{2}+1\}\)
become equal to \(1\) and are removed. We denote the initial seed of $\CN$ by $\{ (\Delta^r (k, \In))_{k \in e(\In) \cup [-n,-1]}, \tilde{A^r_{\In}}\}$. We now give the connection between the $\rn$ initial cluster variables and certain fundamental one-skeleton paths.

\subsection{Initial seed in fundamental one-skeleton paths}\label{subsection: initial seeds}
Each initial cluster variable corresponds to an MV polytope, as shown by the following theorem.
\begin{theorem}[Lemma 4.5.1 \cite{initial1}]\label{thm: crystal sequence for initial seeds}
    The cluster monomials are in the semicanonical basis of $\CN$. In particular, the cluster variables in the initial seed $\{ (\Delta^r (k, \In))_{k \in [-n,-1]\cup e(\In)}, \tilde{A^r_{\In}}\}$ are determined by the following equality:
    \begin{align}\label{eq: general seed formula}
        1=\begin{cases}
        \tilde{f}^{max}_{i_\rn}\cdots \tilde{f}^{max}_{i_1}(\Delta^r (k, \In)) & \textit{ if }
k\in [-n,-1],\\
 \tilde{f}^{max}_{i_\rn}\cdots \tilde{f}^{max}_{i_{k+1}}(\Delta^r (k, \In)) & \textit{ if }
k\in e(\In),
        \end{cases}
    \end{align}
    where $\tilde{f}_i$ are crystal operators for the set of semicanonical basis, which is isomorphic to $B(\infty)$. $\tilde{f}^{max}_i(v)$ is defined by $\tilde{f}^{max}_i(v)=\tilde{f}^{m}_i(v)$, where $m$ is the largest integer such that $\tilde{f}^{m}_i(v)\neq 0$. 
\end{theorem}

We define $t(k)=\min \{h>k \mid h\notin e(\In)\}$ for any $k\in e(\In)$. We denote the set \begin{align}
    \{x \text{ is in the set of the semicanonical basis} \mid x=f'_{i_1}\cdots f'_{i_m}(y), i_k\in[n]\text{ for all }k\in [m] \}, \nonumber
\end{align}
by $C_y$ for any $y$ in the set of the semicanonical basis. It has a $G$-crystal structure. Using the Lusztig datum of the initial cluster variables and of fundamental one-skeleton paths in the $G$-crystal $B(\infty)$, we have the following theorem:

\begin{theorem}\label{thm: path parametrization}
  We consider the initial seed $\{(\Delta^r (k, \In))_{k \in e(\In) \cup [-n,-1]}, \tilde{A^r_{\In}}\}$ of the cluster algebra $\CN$. The cluster variables $x_t$ are parametrized by the oriented edges of the fundamental alcove $\Delta$: \[
  E=\{\omega_i,\omega_j-\omega_k\}_{1\leq i\leq n,1\leq j<k\leq n}.
  \]
  
  The parameterization is given by a bijection $\beta$ from $(\Delta^r (k, \In))_{k \in e(\In) \cup [-n,-1]}$ to $E$ : \begin{align}\label{eq: bijection}
      \beta(\Delta^r (k, \In))=
      \begin{cases}
          \omega_{n+1-i} & \textit{ if } k=-i<0,\\
         \omega_{i_{t(k)}-i_{k}}-\omega_{i_{t(k)}}  & other wise.
      \end{cases}
  \end{align}
  
  Moreover, $\beta$ induces an isomorphism of crystals between $C_{\Delta^r (k, \In)}$ and $B_{\Pol(\beta(\Delta^r (k, \In)))}$.
 \end{theorem}

\begin{proof}
    The bijectivity of $\beta$ is given by Theorem~\ref{thm: initial seed functions}. We just need to prove the isomorphisms between crystals. Using Equation~\eqref{eq: general seed formula} and the expression \begin{align}
        \bar{\mathbf{i}}=(i_{-n}, \ldots, i_{-1}, i_{1}, \ldots, i_{\rn}) = (-n, \ldots, -1, 1, \cdots, n,\cdots,1), \nonumber
    \end{align}
    
    we can express $\Delta^r (k, \In)$ in the G-crystal of semicanonical basis as follows (see Lemma 5.5 in \cite{initial1}):\begin{align}\label{eq: explicit A_n crystal sequence}
        \Delta^r (k, \In)=\begin{cases}
          \tilde{e}_i\cdots \tilde{e}_n\tilde{e}_{i-1}\cdots \tilde{e}_{n-1}\cdots \tilde{e}_1\cdots \tilde{e}_{n+i-1}(1) \ \ \ \ \ \ \ \ \ & \textit{if } k=-(n+1-i)<0,\\
        \tilde{e}_{i_{t(k)}-i_{k}}\cdots \tilde{e}_{i_{t(k)}-1}\tilde{e}_{i_{t(k)}-i_{k}-1}\cdots \tilde{e}_{i_{t(k)}-2}\cdots \tilde{e}_1\cdots \tilde{e}_{i_{k}}(1)   & other wise.
      \end{cases}
    \end{align}
    
Here $1$ is the constant function in the $G$-crystal of semicanonical basis. We calculate $\Pol(\omega_i)$ in the $G$-crystal of MV polytopes as follows:\begin{align}\label{eq: crystal operators for w_i}
        \Pol(\omega_i)=e_{i}\cdots e_ne_{i-1}\cdots e_{n-1}\cdots e_1\cdots e_{n+i-1}(\bullet).
    \end{align}
    And we compute $\Pol(\omega_i-\omega_j)$ for all $i<j$: 
    \begin{align}\label{eq: crystal operators for w_i-w_j}
    \Pol(\omega_i-\omega_j)=\tilde{e}_i\cdots \tilde{e}_{j-1}\tilde{e}_{i-1}\cdots \tilde{e}_{j-2}\cdots \tilde{e}_1\cdots \tilde{e}_{j-i}(\bullet).
\end{align}

    Using the unique isomorphism $\psi$ of crystals between semicanonical basis and $\MV$, we compare Equation~\eqref{eq: explicit A_n crystal sequence}, ~\eqref{eq: crystal operators for w_i} and~\eqref{eq: crystal operators for w_i-w_j} to conclude that $\psi(\Delta^r (k, \In))=\Pol(\beta(\Delta^r (k, \In)))$ for all $k \in e(\In) \cup [-n,-1]$. Thus $C_{\Delta^r (k, \In)}\cong B_{\Pol(\beta_n(\Delta^r (k, \In)))}$ as G-crystals.   
\end{proof}

The definition of the exchange quiver induces the following corollary of Theorem~\ref{thm: path parametrization} via the bijection $\beta$:

\begin{corollary}\label{coro: exchange matrix}

When $t=\omega_i$ is an edge adjacent to the original point $0$ for some $1\leq i\leq n$, the cluster variable $\beta^{-1}(t)$ is frozen. When $t=\omega_j-\omega_k$ is an edge not adjacent to the original point $0$ for some $1\leq j<k\leq n$, $\beta^{-1}(t)$ is mutable.Let
\[
E_{\mathrm{mut}}:=\{\omega_j-\omega_k\mid 1\leq j<k\leq n\}
\]
be the set of mutable vertices. The extended exchange matrix
\[
\widetilde B=(b_{t,t'})_{t\in E,\ t'\in E_{\mathrm{mut}}}
\]
is given by the following rule:
 \begin{align}
        b_{tt'} =
\begin{cases}
1 & \text{if} \ \ t=\omega_i \ \ \text{and} \ \ t'=\omega_i-\omega_n \ \ \text{for} \ \ 1\leq i\leq n-1, \\
1 & \text{if} \ \ t=\omega_j-\omega_k \ \ \text{and} \ \ t'=\omega_j-\omega_{k-1} \ \ \text{for} \ \ 1\leq j<k-1<k\leq n-1, \\
-1 & \text{if} \ \ t=\omega_j-\omega_k \ \ \text{and} \ \ t'=\omega_{j+1}-\omega_k \ \ \text{for} \ \ 1\leq j<k-1<k\leq n-1, \\
0 & \text{else}.
\end{cases} \nonumber
    \end{align}
    
\end{corollary}
\begin{proof}
    This is just the reformulation of the exchange matrix in \cite{proj1} using the parameterization $\beta$ given in Theorem~~\ref{thm: path parametrization}.
\end{proof}
From Corollary~\ref{coro: exchange matrix}, if two cluster variables in the initial seed correspond to two non-adjacent fundamental one-skeleton paths $t$ and $t'$ in $E$, the entry $b_{tt'}$ is always equal to $0$ in the exchange matrix.

\section{Comultiplication structure of one-skeleton paths}\label{sec: comulti}

\subsection{Comultiplication rules for string modules}\label{subsection: comulti fpr strings}

The set $\CN$ has a Hopf algebra structure, where the multiplication is the product of functions and the comultiplication is induced by the group multiplication structure of $N$. We use $x_{i,j}$ to represent the coordinate function which sends $X$ to $X_{i,j}$ for any $X\in N$. We have $\CN=\mathbb{C}[x_{i,j}]_{1\leq i<j\leq n+1}$. In this section, we denote the semicanonical basis of $\CN$ by the notation $\mathcal{B}$. By Definition~\ref{def: algebraic one-skeleton map}, the semicanonical basis $\mathcal{B}$ induces a one-skeleton algebraic map $\phi_{\mathcal{B}}$.

\begin{proposition}\label{prop: string path}
    For any coordinate function $x_{i,j}$, where $1\leq i< j\leq n+1$, there exists a fundamental one-skeleton path $p_{i,j}\in P_f$ such that $\phi_{\mathcal{B}}(p_{i,j})=x_{i,j}$. Such a fundamental one-skeleton path is called a $\textbf{string path}$.
\end{proposition}

\begin{proof}
    For any $1\leq i< j\leq n+1$, the coordinate function $x_{i,j}$ is induced by the following indecomposable module $M_{i,j}$ of $\Lambda_n$:
    \begin{align}
        M_{i,j} =
\xymatrix{
  0\cdots \ar@<0.5ex>[r]^{0} & 
  \C_{i} \ar@<0.5ex>[l]^{0} \ar@<0.5ex>[r]^{1} & 
  \cdots \ar@<0.5ex>[l]^{0} \ar@<0.5ex>[r]^{1} & 
  \C_j \ar@<0.5ex>[l]^{0} \ar@<0.5ex>[r]^{0} & 
  \cdots 0 \ar@<0.5ex>[l]^{0},
} \nonumber
    \end{align}
    where $(M_{i,j})_k\cong \C$ for all $i\leq k\leq j-1$.\par
    Thus we have \begin{align}
        x_{i,j}=\tilde{e}_i\cdots \tilde{e}_{j-1}(1) \nonumber
    \end{align}
    for the $G$-crystal structure of $\mathcal{B}$. Choosing the subquiver $Q_{i,j}$ of the $(n+2-j)$-chain as follows:
    \begin{align}
        \xymatrix{
  \omega_i-\omega_{i-1}-\omega_{j}\ar[r]^-{i} & 
  \cdots \ar[r]^-{j-1} & -\omega_{j-1} .
} \nonumber
    \end{align}
    We have $\Pol(\omega_i-\omega_{i-1}-\omega_{j})=e_i\cdots e_{j-1} (\bullet)$ in the $G$-crystal $\MV$. Take $p_{i,j}=\omega_i-\omega_{i-1}-\omega_{j}$. We have $\phi_{\mathcal{B}}(p_{i,j})=x_{i,j}$ by the unique isomorphism of $G$-crystal from $\mathcal{B}$ to $\MV$. Here we take the convention $\omega_0=\omega_{n+1}=0$.
\end{proof}

\begin{remark}\label{remark: string modules}
    The indecomposable modules $M_{i,j}$ defined in the proof of Proposition~\ref{prop: string path} are called \textbf{string modules} of type $A_n$. 
\end{remark}

Using the comultiplication formula on $\CN$, we have 
\begin{align}
    \Delta(x_{i,j}) = \sum_{k=i}^{j} x_{i,k} \otimes x_{k,j}, 
\qquad 1 \leq i < j \leq n+1. \nonumber
\end{align}

Proposition~\ref{prop: string path} implies that the map $\phi_{\mathcal{B}}$ is surjective. Let $\mathcal{J}$ denote the kernel of $\phi_{\mathcal{B}}$, we have $\Fp/\mathcal{J}\overset{\widetilde{\phi_{\mathcal{B}}}}{\cong} \CN$ as an algebra. We denote the algebra $\Fp/\mathcal{J}$ by $\mathcal{A}_\mathcal{B}$, which is an algebra of one-skeleton path with relations. Transferring the comultiplication of $\CN$ to $\ASC$ via Proposition~\ref{prop: string path}, we can obtain a comultiplication structure on $\ASC$ as follows:
\begin{align}\label{eq: path comultiplication}
     \Delta(p_{i,j}) = \sum_{k=i}^{j} p_{i,k} \otimes p_{k,j}, 
\qquad 1 \leq i < j \leq n+1. 
\end{align}

This comultiplication structure of one-skeleton paths relates closely to the decomposition of the root space inspired by the following geometric explanation of the tensor product structure. We use $A * B$ to denote the concatenation of the two MV polytopes $A$ and $B$, that is, the union of $A$ and $B + \wt(A)$. We have the following corollary:

\begin{corollary}\label{coro: mv polytopes}
    For any $1\leq i<j\leq n$, the MV polytope $\Pol(p_{i,j})$ is a $(j-i)$-dimensional MV polytope. For any $k,h\in [n]$, $p_{i,k} \otimes p_{h,j}$ appears in $\Delta(p_{i,j})$ with a non zero coefficient if and only if $\wt(\Pol(p_{i,k}))+\wt(\Pol(p_{h,j}))=\wt(\Pol(p_{i,j}))$ and $\Pol(p_{i,k})*\Pol(p_{h,j})\subset \Pol(p_{i,j})$.
\end{corollary}
\begin{proof}
   The MV polytope $\Pol(p_{i,j})$ is computed using Proposition~\ref{prop: string path}. The expression~\eqref{eq: w_i-w_k formula} shows that the polytope $\Pol(p_{i,j})$ is a $(j-i)$-dimensional. The last part of the statement is given by Corollary~\ref{coro: mv polytopes}.
\end{proof}

Corollary~\ref{coro: mv polytopes} tells us that finding all terms of the comultiplication of a one-skeleton path $p_{i,j}$ is just the same as finding all decompositions of the associated (j-i)-dimensional root space containing the MV polytope $\Pol(p_{i,j})$ into a direct sum of two smaller root spaces containing a $*$ decomposition of $\Pol(p_{i,j})$. \par

The concatenation of one-skeleton paths corresponds to products of functions via the one-skeleton algebraic map $\phi_{\mathcal{B}}$. We can compute the comultiplication of functions in $\CN$ using the concatenations of one-skeleton paths in $\ASC$. Moreover, this method refines the MV polytope model for the comultiplication: while the latter only givens the existence of certain terms (e.g., the leading term) in the comultiplication expression without determining their explicit coefficients, our approach gives a more precise computation. See \cite{ANDKOGAN} for an example determining the leading term in the comultiplication. The following example shows how we can get comultiplication coefficients using the path model.  For the Duistermaat--Heckman measure model in \cite{MVHD}, let $\mu$ be a measure on $X$ whose support is an MV polytope $P_{\mu}$. For any $\lambda\neq 0$, the support of $\lambda\cdot \mu$ is still $P_{\mu}$. Consequently, these coefficients cannot be detected within the MV polytope model. To obtain the comultiplication coefficients, we must use the Duistermaat--Heckman measure rather than relying only on the MV polytope \cite{MVHD}. \par

\begin{example}\label{example: computation in A_2}
 We consider the case in which $G$ is of type $A_2$. The unipotent radical $N$ is of the form 
 \[
N=\left\{
\begin{bmatrix}
1 & x_{12} & x_{13} \\
0 & 1 & x_{23} \\
0 & 0 & 1
\end{bmatrix}
\ \middle|\ x_{ij}\in\mathbb{C},\; 1\leq i<j\leq 3
\right\}.
\]
The algebra $\CN$ is generated by the set $\{x_{12},\,x_{23},\,x_{13}\}$ as an algebra. Here we have three coordinate functions $x_{12},\,x_{23}$ and $x_{13}$. We have $\phi_{\mathcal{B}}(\omega_1-\omega_2)=x_{1,2}$, $\phi_{\mathcal{B}}(\omega_2-\omega_1)=x_{2,3}$ and $\phi_{\mathcal{B}}(\omega_1)=x_{1,3}$. So the comultiplication on $\ASC$ is generated by the following relations: 
\[
\Delta(\omega_1-\omega_2) 
= (\omega_1-\omega_2)\otimes 1 + 1\otimes (\omega_1-\omega_2), \] \[
\Delta(\omega_2-\omega_1) 
= (\omega_2-\omega_1)\otimes 1 + 1\otimes (\omega_2-\omega_1), \]  \[ 
\text{ and } \Delta(\omega_1) 
= \omega_1\otimes 1 + 1\otimes \omega_1 
   + (\omega_1-\omega_2) \otimes (\omega_2-\omega_1).
\]

To compute $\Delta(x_{1,3}^2)$ in $\CN$, we can use the one-skeleton model:
\begin{align}\label{eq: explicit coefficient in computation}
   \Delta(x_{1,3}^2) 
   &= \widetilde{\phi}_{\mathcal{B}}^{-1}(\Delta(\omega_1*\omega_1)) \nonumber \\ 
   &=\widetilde{\phi}_{\mathcal{B}}^{-1}(\Delta(\omega_1)\cdot \Delta(\omega_1)) \nonumber \\
   &=\widetilde{\phi}_{\mathcal{B}}^{-1}([\omega_1\otimes 1 + 1\otimes \omega_1 
   + (\omega_1-\omega_2) \otimes (\omega_2-\omega_1)]\cdot \nonumber \\ 
   & \ \ \ \ \  [\omega_1\otimes 1 + 1\otimes \omega_1 
   + (\omega_1-\omega_2) \otimes (\omega_2-\omega_1)]) \nonumber \\ 
    &=\widetilde{\phi}_{\mathcal{B}}^{-1}(\omega_1*\omega_1\otimes 1 + 1\otimes \omega_1*\omega_1+2\omega_1\otimes \omega_1 \nonumber \\ 
   & \ \ \ + \omega_1 *(\omega_1-\omega_2) \otimes (\omega_2-\omega_1)+(\omega_1-\omega_2)\otimes \omega_1 * (\omega_2-\omega_1)  \nonumber \\ 
    & \ \ \ + (\omega_1-\omega_2)*\omega_1  \otimes (\omega_2-\omega_1)+(\omega_1-\omega_2) \otimes \omega_1 * (\omega_2-\omega_1)  \nonumber \\ 
     & \ \ \ + (\omega_1-\omega_2)*(\omega_1-\omega_2) \otimes  (\omega_2-\omega_1)*(\omega_2-\omega_1)) \nonumber \\ 
     &=\widetilde{\phi}_{\mathcal{B}}^{-1}(\omega_1*\omega_1\otimes 1 + 1\otimes \omega_1*\omega_1+2\omega_1\otimes \omega_1 \nonumber \\ 
   & \ \ \ + 2\omega_1 *(\omega_1-\omega_2) \otimes (\omega_2-\omega_1)+2(\omega_1-\omega_2)\otimes \omega_1 * (\omega_2-\omega_1)  \nonumber \\ 
     & \ \ \ + (\omega_1-\omega_2)*(\omega_1-\omega_2) \otimes  (\omega_2-\omega_1)*(\omega_2-\omega_1))  \\ 
      &= x_{1,3}^2\otimes 1+1\otimes x_{1,3}^2+2x_{1,3}\otimes x_{1,3}+ 2x_{1,3}x_{1,2}\otimes x_{2,3}+ 2x_{1,3}x_{2,3}\otimes x_{1,2}+x_{1,2}^2\otimes x_{2,3}^2. \nonumber
\end{align}
\end{example}

Notice that Equation~\eqref{eq: explicit coefficient in computation} gives a $\Z$-linear combination of tensor products of one-skeleton paths. The coefficients just give the corresponding comultiplication coefficient in $\CN$. Motivated by this, we expect to give the explicit Littlewood-Richardson coefficient of tensor product of representations of $G$ using the one-skeleton model.

\subsection{An intrinsic comultiplication structure on galleries in Coxeter complexes}\label{subsection: geometric comultiplication}

In this section, we use the notation of the Coxeter complex and folded galleries defined in Section~\ref{sec: pre}. We show that the comultiplication structure of $\CN$ coincides with an intrinsic comultiplication structure on galleries, which arises in the study of the Hopf algebra structure of subword complexes in \cite{BC}. The set $\tilde{\mathcal{C}}$ can be equipped with a category structure. This category has a Hall-like algebra isomorphic to the Hopf algebra defined by Bergeron and Ceballos in \cite{BC}. We give the explicit construction of this category in \cite{GorskyLi}. Any string path of type $A_n$ corresponds to a (folded) gallery in the Coxeter complex $\Sigma(W)$, where $W$ is the Coxeter group of type $A_n$.\par

We introduce a comultiplication structure for galleries in the Coxeter complex $\Sigma(W)$ using projections to subroot systems. This is given by the proto-exact structure of the subquiver category. We prove that this geometric comultiplication structure corresponds to the comultiplication of functions in $\CN$. Moreover, this correspondence works for all functions in $\CN$ induced by indecomposable modules of any Dynkin quiver. This correspondence provides a geometric method to calculate the comultiplication coefficients for a large family of functions in the semicanonical basis $\mathcal{B}$. Since Definition~\ref{def: algebraic one-skeleton map} works for any perfect basis, we have the same results for any perfect basis of $\CN$ using the one-skeleton algebraic map with respect to this perfect basis. \par

\begin{definition}
    Let $X$ be a quadruple $(W_X,Q_X,\pi_X,I_X)\in \tilde{\mathcal{C}}$. Suppose that $Q_X=s_{i_1}\cdots s_{i_{n}}$, where $n=n_X$ and $W_X$ is generated by $\{s_i\}_{i\in [n]}$. For any $l\in [n_X]$, we have an associated root 
 \begin{align}\label{eq: root function}
\rx_X(l):=\prod_{j\in [l-1] \backslash I_X}s_{i_j}(\alpha_{i_l}).     
 \end{align}
We get a function $\operatorname{r}_X$ from $[n_X]$ to the set of roots associated with $W_X$ by Equation~\eqref{eq: root function}, which is called the \textbf{root function} of $X$.
Let $R(X):=\{\rx _X(i)\}_{i\in I_X}$ denote the \textbf{root configuration} of $X$ and $r(X):=\{\rx_X(i)\}_{i\in [n_X]}$ denote the set of all such roots of $X$. A quadruple $X$ is \textbf{irreducible} if $V_X=\operatorname{span}R(X)$. For any subset $H$ of $[n_X]$, we denote the vector space spanned by $\{r_X(i) \mid i\in H\}$ by $\V(H)$.
\end{definition}

\begin{definition}\label{def: flat}
      A \textbf{flat} of a quadruple $X\in \tilde{\mathcal{C}}$ is a subset $F\in [n_X]$ with an associated subspace $\V(F)$ of $V_X$ such that \begin{align}
        \{\rx_X(i)\}_{i\in F}=r(X) \cap \V(F).   \nonumber
    \end{align}
    Conversely, any subspace $V$ of $V_X$ gives a flat $\operatorname{F}(V)=\{i\in[n_X] \mid \rx_X(i)\in r(X) \cap V\}$. We have $\F(\V(F))=F$ for any flat $F$ and $\V(\F(V))=V$ for any subspace $V$ of $V_X$. \par
    We say that a flat $F'$ is a \textbf{subflat} of another flat $F$ if $F'\subseteq F$. 
\end{definition}

Given a flat $F$ of $X$, the subspace $\V(F)$ of $V_X$ contains a natural root system
\[
\Phi_F = \Phi_F^+ \sqcup \Phi_F^- 
\]
where $\Phi_F$, $\Phi_F^+$, $\Phi_F^-$ are the restrictions of $\Phi_X$, $\Phi_X^+$, $\Phi_X^-$ to $\V(F)$ respectively. The fact that the intersection of a root system $\Phi$ and a subspace $V$ is again a root system with simple roots contained in $\Phi^+\cap V$ is a non-trivial result by Dyer in \cite{Dyer} and by Deodhar in \cite{Deodhar}. Suppose that $Q_X=s_{i_1}\cdots s_{i_m}$. Set $F=\{j_1,\cdots ,j_r\}$ where $j_l<j_{l+1}$ for any $l<r$.  Define $\beta_F = (\beta_1, \ldots, \beta_{r})$ as the list of roots
\[
\beta_k := \prod_{j\in[B_k]}s_{i_j} (\alpha_{i_{j_k}}),
\]
where $B_k := ([j_k-1] \setminus I)\setminus F$ is the set of positions on the left of $j_k$ in the complement of $I$ which are not in $F$.

\begin{lemma}[Lemma 2.10 and Theorem 2.11 in \cite{BC}]\label{lem: subobject induced by flats}
     Any flat $F$ of an quadruple $X\in \tilde{\mathcal{C}}$ induces a quadruple $X_F\in \tilde{\mathcal{C}}$, such that $V_{X_F}=\V(F)$ and $n_{X_F}=|F|$. Simple roots of $\Phi_{X_F}$ are given by $\beta_F$. \par
     If we identify $F$ with $[n_{X_F}]$ via the unique order-preserving bijection denoted by $b_F$, we have $r_X(i)=r_{X_F}(b_F(i))$ for any $i\in F$. \par
     We say that a flat $F$ is \textbf{irreducible} if $X_F$ is irreducible. In this case, we say that $X_F$ is an \textbf{(irreducible) subquadruple} of $X$. 
\end{lemma}

Flats of a quadruple $X$ give a discrete description of certain subspaces of $V_X$. In the rest of this section, we consider quadruples in $\CC$, where $\CC$ is the collection of all quadruples $X$ such that $W_X$ is a finite product of Coxeter groups of type $A_n$, $R(X)$ is a \textbf{basis} of $V(X)$ and $\pi_X$ is the longest element in $W_X$. For a quadruple $X\in\CC$, we have a precise description of flats of $X$ and quadruples induced by flats \cite{GorskyLi}:

\begin{enumerate}
    \item Any subset $J\subseteq I_X$ gives an irreducible flat $F_J$ such that $\V(F_J)=\V(F)$. Conversely, any irreducible flat $F$ of $X$ is given by $F\cap I_X\subseteq I_X$ in this way.
    
    \item Given an irreducible flat $F$ of $X$, the subquadruple $X_F$ is uniquely determined by the subset $I(F):=F\cap I_X$ of $I_X$. Moreover, we have $X_F\in \CC$ since reflection subgroups of a finite Coxeter group are conjugate to parabolic subgroups. Conversely, any subset $J$ of $I_X$ gives an irreducible subquadruple $X(J):=X_{\F(\V(J))}$ of $X$. 
\end{enumerate}

 We associate a quiver to any $X\in \CC$ as follows:
\begin{definition}[Definition 3.7 in \cite{GorskyLi}]\label{def: root configuration quiver}
 For any $X=(W,Q,\pi,I)\in \CC$, we denote the length of $Q$ by $n$ and the cardinal of $I$ by $m$. The root function of $X$ is denoted by $\rx$. The \textbf{root configuration quiver} $\Gamma_X$ of $X$ is a labeled quiver without multiply arrows defined as follows:
 \begin{enumerate}
     \item The vertex set of $\Gamma_X$ is labeled by $I$.
     \item There exists an arrow from $i$ to $j$ in $\Gamma_X$ if and only if \begin{align}\label{eq: definiton of order}
         (i-j)\cdot\langle \rx (i), \rx(j)^{\vee}\rangle >0.
     \end{align}
 \end{enumerate}
\end{definition}

For any irreducible subquadruple $X_F$ of $X$, we label letters of $Q_{X_F}$ using $F$ instead of $[|F|]$ to simplify the notations. In this way, we can view $\Gamma_{X_F}$ as a subquiver of $\Gamma_X$.
\begin{proposition}[Proposition 3.13 in \cite{GorskyLi}] \label{prop:subquivers_subobjects}
    For any irreducible subquadruple $X_F$ of $X$, $\Gamma_{X_F}$ is a subquiver of $\Gamma_X$. Conversely, any subquiver of $\Gamma_X$ is of the form $\Gamma_{X_F}$ for an irreducible flat $F$.
\end{proposition}

Now we can define the subquiver category $\mathcal{S}_X$:

\begin{definition} \label{def:subquiver_category}
    For any $X\in \CC$, we define the \textbf{subquiver category} $\mathcal{S}_X$ to be the category whose objects are finite disjoint unions of labeled quivers $\bigsqcup^t_{i=1} \Gamma_{X_i}$, where $X_i=X(J_i)$ for a set $J_i\subseteq I$. Morphisms are given by label preserving partial isomorphisms of quivers. An irreducible subquadruple $X_F$ of $X$ given by an irreducible flat $F$ is \textbf{indecomposable} if the root configuration quiver $\Gamma_{X_F}$ associated to $X_{F}$ is connected. By abuse of notation, we also call the object $\Gamma_{X_F}$ an indecomposable object in $\mathcal{S}_X$.
\end{definition}

 Recall that a label preserving partial isomorphism is a diagram \begin{align}
        Q \hookleftarrow P \hookrightarrow Q' \nonumber
    \end{align} given by an $I$-labeled quiver $P$, which is a subquiver of both $Q$ and $Q'$. Compositions of label preserving partial isomorphisms 
        $Q \hookleftarrow P \hookrightarrow Q'$ and $Q' \hookleftarrow P' \hookrightarrow Q''$ is given by the intersection of $P$ and $P'$ in $Q'$.

\begin{remark}\label{remark: disjoint union}
    The disjoint union of quivers gives a symmetric monoidal structure on $\mathcal{S}_X$. We call $\Gamma_1 \sqcup \Gamma_2$ the \textbf{direct sum} of $\Gamma_1$ and $\Gamma_2$ in $\mathcal{S}_X$. The direct sum here is neither the product or the coproduct. See \cite{GorskyLi} for details. 
\end{remark}
 The subquiver category is well defined. Every object in $\mathcal{S}_X$ can be written as a direct sum of objects of the form $\Gamma_{X(J)}$, $J\subseteq I$. The identity map is the identity map from a quiver to itself. The associativity of morphisms is induced by the associativity of intersections of sets. The zero object of $\mathcal{S}_X$ is the empty quiver $\Gamma_{X(\emptyset)}$. \par
 Notice that $\Gamma\in \operatorname{Mor}(\Gamma_1,\Gamma_2)$ is a monomorphism if and only if $\Gamma=\Gamma_1$. Symmetrically, $\Gamma\in \operatorname{Mor}(\Gamma_1,\Gamma_2)$ is an epimorphism if and only if $\Gamma=\Gamma_2$. \par

Any object of the form $\Gamma_{X(J)}$, $J\subseteq I$ in $\mathcal{S}_X$ corresponds to a folded gallery in the Coxeter complex of $\Sigma(W_{X(J)})$. We use the sequence of types to represent a folded gallery starting at the fundamental chamber $id$ defined in Section~\ref{sec: pre}.

\begin{proposition}\label{prop: gallery-subword correspondence}
    Given a Coxeter group $W$ with the associated Coxeter complex $\Sigma(W)$, any object $X\in \CC$ such that $W_X=W$ corresponds to a unique folded gallery $g_X$ in $\Sigma(W)$ starting at the fundamental chamber $id$.
\end{proposition}

\begin{proof}
    The root function $\rx_X$ determines a folded gallery $g_X$ in $\Sigma(W)$ starting at the fundamental chamber $id$ as follows: \begin{align}
        g_X=([(-1)^{\delta_I(1)}\rx_X(1),\cdots ,(-1)^{\delta_I(n)}\rx_X(n)],I), \nonumber
    \end{align}
    where $\delta_I$ is the indicator function of $I$ and $n=n_X$. Here $I=I_X$ represents the reflection positions of $g_X$. \par
    Conversely, given a folded gallery
    \begin{align*}
        g=([(-1)^{\delta_I(1)}r_1,\cdots ,(-1)^{\delta_I(n)}r_n],I),
    \end{align*}
    we have $r_1$ is a simple root since $g$ starts at the fundamental chamber $id$. \par
    The folded gallery $g$ gives $X=(W_X,Q_X,\pi_X, I_X)\in \CC$ as follows:\begin{enumerate}
        \item We take $W_X=W$, $n_X=n$ and $I_X=I$.
        \item Let the first letter $s_{i_1}$ of $Q_X$ be the simple reflection corresponding to the simple root $r_1$. Then for any $l\in [n]$, we define the simple reflection $s_{i_l}$ associated with the simple root $\alpha_{i_l}$ recursively by the condition \begin{align}
            r_l=\prod_{j\in [l-1] \backslash I_X}s_{i_j}(\alpha_{i_l}). \nonumber
        \end{align}
        \item The element $\pi_X$ is the product $\prod_{j\in [n] \backslash I_X}s_{i_j}$.
    \end{enumerate}
\end{proof}

\begin{definition}\label{def: canonical gallery}
Using the notations in the proof above, elements in $[n_X]\backslash I_X$ are called \textbf{traversing positions} of the gallery $g_X$ and elements in $I_X$ are called \textbf{reflection positions} of the gallery $g_X$. The gallery $g_X$ is called the \textbf{canonical gallery} associated to $X$. And $n=n_X$ is the \textbf{length} of the gallery $g_X$. 
\end{definition}
In this way, any irreducible quadruple $X(J)$ in $\mathcal{S}_X$ can be viewed as a folded gallery in the Coxeter complex $\Sigma(W_{X(J)})$. Objects in $\mathcal{S}_X$  can be viewed as formal unions of folded galleries in Coxeter complexes.
\begin{theorem}\label{thm: subobject-projection}
    Given $X\in \CC$ with the associated canonical gallery $g_X$, any irreducible subquadruple of $X$ is given by the geometric projection of $g_X$ to a sub-Coxeter complex.
\end{theorem}

\begin{proof}
    Any irreducible subquadruple $X_F$ of $X$ is given by an irreducible flat $F$ of $X$. We choose the sub-Coxeter complex $\Sigma(X_F)$ given by $\Phi_{X_F}$, and denote the geometric projection from the set of folded galleries in $\Sigma(X)$ to the set of folded galleries in $\Sigma(X_F)$ by $\operatorname{p}_F$. Suppose that \begin{align}\label{eq: 1ingallery}
    g_X=([(-1)^{\delta_I(1)}\rx_X(1),\cdots ,(-1)^{\delta_I(n)}\rx_X(n)],I)=([g_1,\cdots , g_n],I).
    \end{align}
    We have \begin{align}\label{eq: 2in gallery}
        \operatorname{p}_F(g_X)= ([g_{i_1},\cdots, g_{i_m}], \bx_F(I\cap F)),
    \end{align}
    where $i_1<\cdots <i_m$ are all elements in $I\cap F$. Comparing Equation~\eqref{eq: 1ingallery} and Equation~\eqref{eq: 2in gallery}, we have $g_{i_k}=(-1)^{\delta_{I\cap F}(i_k)}\rx_X(i_k)$ for any $i_k\in I\cap F$. \par
    By Lemma~\ref{lem: subobject induced by flats}, the root function of $X_F$ is just the restriction of the root function of $X$ on $F$. We have \begin{align}
        g_{X_F} &= ([(-1)^{\delta_{I\cap F}(i_1)}\rx_{X_F}(1),\cdots ,(-1)^{\delta_{I\cap F}(i_m)}\rx_{X_F}(m)],\bx_F(I\cap F))\nonumber \\
        &= ([(-1)^{\delta_{I\cap F}(i_1)}\rx_X(i_1),\cdots ,(-1)^{\delta_{I\cap F}(i_m)}\rx_{X}(i_m)],\bx_F(I\cap F)) \nonumber \\
        &= \operatorname{p}_F(g_X). \nonumber
    \end{align}
\end{proof}

Now we can identify quadruples in $\CC$ with folded galleries in the Coxeter complex. Given a quadruple $X\in \CC$, Proposition~\ref{prop: gallery-subword correspondence} and Definition~\ref{def:subquiver_category} provide us with a way to construct a finite category $\mathcal{S}_X$ whose objects are formal unions of folded galleries. We now define a proto-exact structure on $\mathcal{S}_X$ for a standard quadruple $X$ to give a Hall algebra $\mathcal{H}(\mathcal{S}_X)$ for such a category $\mathcal{S}_X$. Moreover, the Hall algebra  $\mathcal{H}(\mathcal{S}_X)$ is isomorphic to the universal enveloping algebra $U(\mathfrak{n}_X^+)$. The multiplication structure in $\mathcal{H}(\mathcal{S}_X)$ is given by counting admissible proto-exact sequences. By the duality between $\CN$ and $U(\mathfrak{n}_X^+)$, the comultiplication structure of $\CN$ is induced from this.

\begin{definition}
    Fix $n\in \mathbb{N}$, we choose the following object $X$:
    \begin{enumerate}
        \item $W_X\simeq S_{n+1}$;
        \item The word $Q_X=s_1\cdots s_ns_1\cdots s_ns_1\cdots s_{n-1}\cdots s_1$;
        \item $I_X=[n]$;
        \item $\pi_X=w_0$ is the longest element in $W_X$.
    \end{enumerate}
\end{definition}

 We call this object $X$ the \textbf{standard quadruple} of type $A_n$ associated to the directed quiver $\Gamma_X=1\rightarrow \cdots \rightarrow n$. Now we fix a standard quadruple $X$ of type $A_n$. Any subset of $I_X$ is a disjoint union of sub-intervals of $[n]$. The following definition is a special case of Definition 3.18 in \cite{GorskyLi}:

 \begin{definition}\label{proto-exact structure}
    For sub-intervals $H\subseteq J$ of $I_X$, the sequence 
    \begin{align}
        \Gamma_{X(H)} \hookrightarrow \Gamma_{X(J)} \nonumber
    \end{align}
    is a \textbf{basic admissible monomorphism} in $\mathcal{S}_X$ if and only if $\{j\in J \mid j>k\}\subseteq H$ for any $k\in H$. The class $\mathcal{M}$ of admissible monomorphisms consists of direct sums of basic admissible monomorphisms. \par
    For sub-intervals $H\subseteq J$ of $I_X$, the sequence 
    \begin{align}
        \Gamma_{X(H)} \twoheadrightarrow \Gamma_{X(J)} \nonumber
    \end{align}
    is a \textbf{basic admissible epimorphism} in $\mathcal{S}_X$ if and only if $\{j\in J \mid j<k\}\subseteq H$ for any $k\in H$. The class $\mathcal{P}$ of admissible epimorphisms consists of direct sums of basic admissible epimorphisms. 
\end{definition}

\begin{theorem}[Theorem 3.19 in \cite{GorskyLi}]\label{thm: proto-exact}
    The pair $(\mathcal{M},\mathcal{P})$ defines a proto-exact structure on $\mathcal{S}_X$, which induces an associative Hall algebra $\mathcal{H}(\mathcal{S}_X)$.
\end{theorem}
\begin{corollary}[Corollary 3.31 in \cite{GorskyLi}] \label{cor:type_A_iso}
    We have an isomorphism of algebra $\psi: \mathcal{H}(\mathcal{S}_X)\to U(\mathfrak{n}_+)$.
\end{corollary}
We denote the subobject of $X$ given by the sub-interval $\{i,\cdots,j\}$ of $I_X$ by $e_{i,j}$ for any $1\leq i\leq j\leq n$. The Hall algebra $\mathcal{H}(\mathcal{S}_X)$ can be described as the $\mathbb{C}$-module generated by $Ind_X:=\{e_{i,j}\}_{1\leq i\leq j\leq n}$ with the multiplication rule
 \begin{align}\label{eq: final0}
        e_{i,j}\cdot e_{h,l}= e_{i,j}+e_{h,l}+\sum_{\begin{array}{c} e_{i,j}\hookrightarrow K\twoheadrightarrow e_{h,l}\\ \text{ admissible} \\ K\in Ind_X \end{array} } K 
    \end{align} 
Here, the sequence $e_{i,j}\hookrightarrow K\twoheadrightarrow e_{h,l}$ is admissible if and only if $i=l+1$ and $K=e_{h,j}$. In this special case, we can realize the isomorphism explicitly.
   The universal enveloping algebra $U(\mathfrak{n}_X^+)$ is the associative algebra
generated by elements
\[
e_1, e_2, \dots, e_n,
\]
subject to the following relations:

\begin{enumerate}
    \item For all $i,j$ with $|i-j|>1$,
    \[
    [e_i, e_j] = 0 .
    \]

    \item For all $i,j$ with $|i-j|=1$ (the Serre relations),
    \[
        [e_i,[e_i,e_j]] = 0,
    \]
    or equivalently,
    \[
        e_i^2 e_j - 2 e_i e_j e_i + e_j e_i^2 = 0.
    \]
\end{enumerate}
We can rewrite the multiplication rule in Equation~\eqref{eq: final0} explicitly as \begin{align}
        e_{i,j}\cdot e_{h,l}=\begin{cases}
            e_{i,j}+e_{h,l}+e_{h,l} , & i=l+1 \\
            e_{i,j}+e_{h,l}, & \text{else.}
        \end{cases} 
    \end{align}
    We now construct a morphism of algebra $\psi:\mathcal{H}(\mathcal{S}_X)\to U(\mathfrak{n}_X^+)$ by sending $e_{i,i}$ to $e_i$ for each $i\in [n]$. 
    Notice that \begin{align}
        [e_{i,i},e_{j,j}]= \begin{cases}
            e_{i,j} , & j=i+1 \\
            0, & \text{else.}
            \end{cases}
    \end{align}
    This $\psi$ is the isomorphism given by Corollary 3.31 in \cite{GorskyLi}.

We use $f^*$ to represent the dual in $\mathcal{H}^*(\mathcal{S}_X)$ for any element $f\in \mathcal{H}(\mathcal{S}_X)$. By the duality between $\CN$ and $U(\mathfrak{n}_X^+)$, we have the following corollary:

\begin{corollary}\label{coro final1}
    The coalgebra morphism defined by: \begin{align}
        \mathrm{d}_n: e^*_{i,j-1}\to x_{i,j}, \ 1\leq i<j\leq n+1 \nonumber
    \end{align}
    is an isomorphism between the coalgebra $\mathcal{H}^*(\mathcal{S}_X)$ and the coalgebra $\CN$.
\end{corollary}

Combining Proposition~\ref{prop: string path} and Corollary~\ref{coro final1}, we can associate a one-skeleton path $p_{i,j}$ in Proposition~\ref{prop: string path} with a folded gallery $e_{i,j-1}$. The comultiplication of $x_{i,j}$ (or $p_{i,j}$) is obtained by computing all pairs of projections of $e_{i,j-1}$ to two subspaces which are in the order induced by the proto-exact structure in Definition~\ref{proto-exact structure}. \par

\begin{example}\label{s_x EXAMPLE}
    Take $n=3$. The standard $X$ of type $A_3$ is of the form \begin{align}
        X=(S_4, s_1s_2s_3s_1s_2s_3s_1s_2s_1, s_1s_2s_3s_1s_2s_1, \{1,2,3\}). \nonumber
    \end{align}
    The objects of $\mathcal{S}_X$ given by sub-intervals of $[3]$ and certain multiplication rules are listed as follows:
\begin{align*}
\{2\} * \{1\} &= \{1\} + \{2\} + \{1,2\}, & \{1\} * \{2\} &= \{1\} + \{2\}, \\[4pt]
\{3\} * \{2\} &= \{2\} + \{3\} + \{2,3\}, & \{2\} * \{3\} &= \{2\} + \{3\}, \\[4pt]
\{1\} * \{3\} &= \{3\} * \{1\}=\{1\} + \{3\}\\[4pt]
\{3\} * \{1,2\} &= \{1,2\} + \{3\} + \{1,2,3\}, &
\{1,2\} * \{3\} &= \{1,2\} + \{3\}, \\[4pt]
\{1\} * \{2,3\} &= \{1\} + \{2,3\} + \{1,2,3\}, &
\{2,3\} * \{1\} &= \{1\} + \{2,3\}.
\end{align*}
Here, we use a set $J$ in $\{\{1\},\{2\},\{3\},\{1,2\},\{2,3\},\{1,2,3\}\}$ to represent the corresponding object $\Gamma_{X(J)}$ of $X$. \par
The canonical gallery $g_X$ of $X$ is given by \begin{align}
    g_X=([-\alpha_1,-\alpha_2, -\alpha_3, \alpha_1 ,\alpha_1+\alpha_2, \alpha_1+\alpha_2+\alpha_3, \alpha_2, \alpha_2+\alpha_3,\alpha_3],\{1,2,3\}). \nonumber
\end{align}
Any set of simple roots determines an irreducible subquadruple $Y$ of $X$, where $W_{Y}$ is the parabolic subgroup of $W_X$ generated by simple reflections corresponding to simple roots in this set. \par 

For example, if we take the set $\{\alpha_2,\alpha_3\}$, then the projection of $g_X$ to the vector space spanned by $\{\alpha_2,\alpha_3\}$ yields a folded gallery \begin{align}
    g=([-\alpha_2, -\alpha_3, \alpha_2, \alpha_2+\alpha_3, \alpha_3],\{1,2\}). \nonumber
\end{align}
The irreducible subquadruple of $X$ corresponding to this gallery is \[Y=(\langle s_2,s_3\mid s_2^2 = s_3^2 = 1, s_2 s_3 s_2 = s_3 s_2 s_3\rangle, s_2s_3s_2s_3s_2, s_2s_3s_2,\{1,2\})\]. \par

All admissible sequences in $\mathcal{S}_X$ with midterm $\Gamma_X$ are obtained by taking all partitions of the sequence $(1,2,3)$ into two parts. The comultiplication of the dual object $\Gamma_X^*$ of $X$ in the dual coalgebra of $\mathcal{H}(\mathcal{S}_X)$ is given by \begin{align}\label{eq: final2}
    \Delta(\Gamma_X^*)=1\otimes \Gamma_X^*+\Gamma_{X(\{1\})}^*\otimes \Gamma_{X(\{2,3\})}^*+\Gamma_{X(\{1,2\})}^*\otimes \Gamma_{X(\{3\})}^*+\Gamma_X^*\otimes 1.
\end{align}
Using the isomorphism $\psi$ in the proof of Corollary~\ref{cor:type_A_iso}, the dual pairing in Corollary~\ref{coro final1} sends $\Gamma_X^*$ to $x_{1,4}$ and sends $\Gamma_{X([i,j])}^*$ to $x_{i,j+1}$ for any sub-interval $[i,j]\subseteq [3]$. This dual pairing identifies the comultiplication rule in the dual coalgebra of $\mathcal{H}(\mathcal{S}_X)$ to the comultiplication rule in $\CN$. \par
For example, Equation~\eqref{eq: final2} is mapped to \begin{align*}
    \Delta(x_{1,4})=1\otimes x_{1,4}+ x_{1,2}\otimes x_{2,4}+ x_{1,3}\otimes x_{3,4}+x_{1,4}\otimes 1
\end{align*} 
in $\CN$.
\end{example}


\section{Examples}\label{sec: examples}

In this section, we give several examples of one-skeleton models explicitly.

\subsection[A2 case]{Type \texorpdfstring{$A_2$}{A2} case }
The simplest nontrivial case is $G=SL_{3}(\mathbb{C})$. In this case, we have two fundamental positive one-skeleton paths $\omega_1$ and $\omega_2$. The associated fundamental one-skeleton paths are listed as follows:
\[
Q_1:\;
\begin{tikzcd}[column sep=0.8em, row sep=1.4em]
\omega_1 \arrow[r,"f_1"] & \omega_2-\omega_1 \arrow[r,"f_2"] & -\omega_2
\end{tikzcd}
\]
\[
Q_2:\;
\begin{tikzcd}[column sep=0.8em, row sep=1.4em]
\omega_2 \arrow[r,"f_2"] & \omega_1-\omega_2 \arrow[r,"f_1"] & -\omega_1.
\end{tikzcd}
\]

There are four prime MV polytopes of type $A_2$, which are given by $\Pol(\omega_1)$, $\Pol(\omega_2)$, $\Pol(\omega_2-\omega_1)$ and $\Pol(\omega_1-\omega_2)$ separately. Note that different subquivers of $j$-chains may induce the same MV polytope. For example, the subquiver  
\[
\begin{tikzcd}[column sep=0.8em, row sep=1.4em]
\omega_2 \arrow[r,"f_2"] & \omega_1-\omega_2
\end{tikzcd}
\]
of $Q_2$ also induces the MV polytope $\Pol(\omega
_2-\omega_1)$.

\subsection[A3 case]{Type \texorpdfstring{$A_3$}{A3} case}

When $G=SL_4(\mathbb{C})$, there are 12 prime MV polytopes listed in \cite{BAI}. We can rediscover all of them using subquivers of $j$-chains:
\[
Q_1:\;
\begin{tikzcd}[column sep=0.8em, row sep=1.4em]
\omega_1 \arrow[r,"f_1"] & \omega_2-\omega_1 \arrow[r,"f_2"] & \omega_3-\omega_2 \arrow[r,"f_3"] & -\omega_3
\end{tikzcd}
\]

\[
Q_2:\;
\begin{tikzcd}[column sep=0.8em, row sep=1.4em]
 & & \omega_3-\omega_1 \arrow[dr,"f_3"] & \\
\omega_2 \arrow[r,"f_2"] & \omega_1+\omega_3-\omega_2 \arrow[ur,"f_1"] \arrow[dr,"f_3"] & & \omega_2-\omega_1-\omega_3 \arrow[r,"f_2"] & -\omega_2 \\
  & & \omega_1-\omega_3 \arrow[ur,"f_1"] &
\end{tikzcd}
\]

\[
Q_3:\;
\begin{tikzcd}[column sep=0.8em, row sep=1.4em]
\omega_3 \arrow[r,"f_3"] & \omega_2-\omega_3 \arrow[r,"f_2"] & \omega_1-\omega_2 \arrow[r,"f_1"] & -\omega_1.
\end{tikzcd}
\]

There are 11 prime MV polytopes induced by fundamental one-skeleton paths. The missing MV polytope is associated with $x_{1,3}x_{3,4}-x_{1,4}$ in the semicanonical basis, which is induced by the indecomposable module :
\[
\begin{tikzcd}
\mathbb{C} 
  \arrow[r, shift left=1ex, "0"] 
& 
\mathbb{C} 
  \arrow[r, shift left=1ex, "id"] 
\arrow[l, shift left=1ex, swap, "id"] 
& 
\mathbb{C}   \arrow[l, shift left=1ex, swap, "0"] 
\end{tikzcd}.
\]

We choose the full subquiver $Q'$ of $Q_2$ by canceling the vertex $-\omega_2$. The convex hull of vertices of $Q'$ is an MV polytope. It is equal to $e'_2e'_1e'_3(\bullet)$ in the $G$-crystal $\MV$, which gives the MV polytope corresponding to $x_{1,3}x_{3,4}-x_{1,4}$ above.

\subsection[A4 case]{Type \texorpdfstring{$A_4$}{A4} case}
When $G$ is of type $A_4$, the $j$-chains are listed as follows:
\[
Q_1:\;
\begin{tikzcd}[column sep=0.8em, row sep=0.8em]
\omega_1 \arrow[r,"f_1"] & \omega_2-\omega_1 \arrow[r,"f_2"] & \omega_3-\omega_2 \arrow[r,"f_3"] & \omega_4-\omega_3 \arrow[r,"f_4"] & -\omega_4
\end{tikzcd}
\]

\[
{\small
Q_2:\;
\begin{tikzcd}[column sep=0.5em, row sep=1em]
  & & \omega_3-\omega_1 \arrow[dr,"f_3"] & & \omega_4-\omega_2 \arrow[dr,"f_4"] & & \\
  \omega_2 \arrow[r,"f_2"] & \omega_1+\omega_3-\omega_2 \arrow[ur,"f_1"] \arrow[dr,"f_3"] & & \omega_2+\omega_4-\omega_1-\omega_3 \arrow[ur,"f_2"] \arrow[dr,"f_4"] & & \omega_3-\omega_2-\omega_4 \arrow[r,"f_3"] & -\omega_3 \\
  & & \omega_1+\omega_4-\omega_3 \arrow[ur,"f_1"] \arrow[dr,"f_4"] & & \omega_2-\omega_1-\omega_4 \arrow[ur,"f_2"] & & \\
  & & & \omega_1-\omega_4 \arrow[ur,"f_1"] & & &
\end{tikzcd}
}
\]
\[
{\small
Q_3:
\begin{tikzcd}[column sep=0.5em, row sep=1em]
     & & & \omega_4-\omega_1 \arrow[dr,"f_4"] & & & \\
  & & \omega_1+\omega_4-\omega_2 \arrow[dr,"f_4"] \arrow[ur,"f_1"] & & \omega_3-\omega_1-\omega_4 \arrow[dr,"f_3"] & & \\
  \omega_3 \arrow[r,"f_3"] & \omega_2+\omega_4-\omega_3 \arrow[ur,"f_2"] \arrow[dr,"f_4"] & & \omega_1+\omega_3-\omega_2-\omega_4 \arrow[ur,"f_1"] \arrow[dr,"f_3"] & & \omega_2-\omega_1-\omega_3 \arrow[r,"f_2"] & -\omega_2 \\
  & & \omega_2-\omega_4 \arrow[ur,"f_2"]  & & \omega_1-\omega_3 \arrow[ur,"f_1"] & & 
\end{tikzcd}
}
\]

\[
Q_4:\;
\begin{tikzcd}[column sep=0.8em, row sep=1.4em]
\omega_4 \arrow[r,"f_4"] & \omega_3-\omega_4 \arrow[r,"f_3"] & \omega_2-\omega_3 \arrow[r,"f_2"] & \omega_1-\omega_2 \arrow[r,"f_1"] & -\omega_1
\end{tikzcd}.
\]

In this case, we can obtain 34 prime MV polytopes among all 40 prime MV polytopes of type $A_4$ from subquivers of $j$-chains. Subquivers of $j$-chains can also give non prime MV polytopes of course. For example, the 3-chain $Q_3$ contains a subquiver $Q'$ as follows:
\[
Q':\;
\begin{tikzcd}[column sep=0.8em, row sep=1.4em]
  & \omega_1+\omega_4-\omega_2 \arrow[dr,"f_4"] & \\
 \omega_2+\omega_4-\omega_3 \arrow[ur,"f_2"] \arrow[dr,"f_4"] & & \omega_1+\omega_3-\omega_2-\omega_4  \\
   & \omega_2-\omega_4 \arrow[ur,"f_2"] &
\end{tikzcd}.
\]
The convex hull of vertices in $Q'$ is isomorphic to the Minkowski sum of the two MV polytopes $e'_2(\bullet)$ and $e'_4(\bullet)$.

 \section*{Acknowledgments}
I am grateful to Stéphane Gaussent for many helpful discussions on path models and MV polytopes, as well as for his valuable suggestions during the revision of this paper. I thank Mikhail Gorsky for insightful conversations on Hall algebras and for his helpful comments on the manuscript. I also thank Petra Schwer and Paul Philippe for discussions related to subword complexes. This work was financially supported by the doctoral contract of Université Jean Monnet Saint-Étienne.

\bibliographystyle{alpha}
\bibliography{references, biblio}

\end{document}